\documentclass[12pt]{amsart}
\usepackage{fullpage}

\usepackage[english]{babel}
\usepackage{enumerate}
\usepackage{amsthm}

\usepackage{amsmath}
\usepackage{amssymb}
\usepackage{amsfonts}

\newtheorem{thm}{Theorem}[section]
\newtheorem*{thm*}{Theorem} 
\newtheorem{lem}[thm]{Lemma}

\newtheorem{prop}[thm]{Proposition}
\newtheorem{cor}[thm]{Corollary}
\theoremstyle{definition}
\newtheorem{Def}[thm]{Definition}
{\newtheorem{ex}[thm]{Example}}
{\newtheorem{rem}[thm]{Remark}}

\newcommand{\N}{\ensuremath{\mathbb{N}}}
\newcommand{\C}{\ensuremath{\mathbb{C}}}

\newcommand{\Q}{\ensuremath{\mathbb{Q}}}
\newcommand{\R}{\ensuremath{\mathbb{R}}}

\newcommand{\del}{\ensuremath{\partial}}

\newcommand{\G}{\ensuremath{\mathcal{G}}}
\newcommand{\Hcal}{\ensuremath{\mathcal{H}}}
\newcommand{\GL}{\operatorname{GL}}
\newcommand{\SL}{\operatorname{SL}}
\newcommand{\SO}{\operatorname{SO}}
\newcommand{\Gm}{\mathbb{G}_m}
\newcommand{\Ga}{\mathbb{G}_a}

\newcommand{\Pcal}{\ensuremath{\mathcal{P}}}
\newcommand{\Bcal}{\ensuremath{\mathcal{B}}}

\newcommand{\Gal}{\ensuremath{\underline{\mathrm{Gal}}}}
\newcommand{\Galf}{\ensuremath{\mathrm{Gal}}}
\newcommand{\Hom}{\operatorname{Hom}}
\newcommand{\Aut}{\operatorname{Aut}}
\newcommand{\Autd}{\underline{\operatorname{Aut}}^\del}

\newcommand{\Spec}{\operatorname{Spec}}

\newcommand{\trd}{\operatorname{trdeg}}
\newcommand{\Frac}{\ensuremath{\mathrm{Frac}}}

\title{Differential Galois Groups over Laurent Series Fields}
\author{David Harbater, Julia Hartmann, and Annette Maier}

\date{January 27, 2015
}

\begin{document}

\begin{abstract} 
In this manuscript, we apply patching methods to give a positive answer to the inverse differential Galois problem over function fields over Laurent series fields of characteristic zero. More precisely, we show that any linear algebraic group (i.e.\ affine group scheme of finite type) over such a Laurent series field does occur as the differential Galois group of a linear differential equation with coefficients in any such function field (of one or several variables).  
\end{abstract}
\maketitle

\section*{Introduction}
Differential Galois theory studies linear homogeneous differential equations by means of their symmetry groups, the differential Galois groups. Such a group acts on the solution space of the equation under consideration, and this furnishes it with the structure of a linear algebraic group over the field of constants of the differential field. In analogy to the inverse problem in ordinary Galois theory, an answer to the question of which linear algebraic groups occur as differential Galois groups over a given differential field provides information about the field and its extensions. 

Classically, differential Galois theory was mostly concerned with one variable function fields over the complex numbers. In that case, the inverse problem is related to the Riemann-Hilbert problem (Hilbert's 21st problem) about monodromy groups of differential equations. In fact, Tretkoff and Tretkoff (\cite{tretkoff}) showed that every linear algebraic group over ${\mathbb C}$ occurs as the differential Galois group of some equation over ${\mathbb C}(x)$, as a consequence of Plemelj's solution to the (modified) Riemann-Hilbert problem (\cite{Plemelj}; see also \cite{Anosov-Bolibrukh}). The solution to the inverse differential Galois problem over $C(x)$ for an arbitrary algebraically closed field~$C$ of characteristic zero was given in \cite{HartCrelle}, building on decades of work by other authors, e.g.\  \cite{kovacic1}, \cite{kovacic2}, \cite{singer}, \cite{mitschisinger1}, \cite{mitschisinger2}. It is worth noting that contrary to ordinary Galois theory, there seems to be no direct way of deducing this from the complex case for general groups (see \cite{singer}).

More recently, there have been results concerning differential Galois theory over non-algebraically closed fields of constants; see, e.g.\ \cite{amano-masuoka}, \cite{Andre}, \cite{Dyc}, \cite{CHP}. (There was actually an earlier attempt, in \cite{Epstein1} and \cite{Epstein2}, but the approach did not yield a full Galois correspondence and seems to have been dropped.) The differential Galois groups in this more general setting are still linear algebraic groups, but in the sense of group schemes rather than of point groups. However, only very limited results are known about the corresponding inverse problem (in particular, see \cite{Dyc}).

In this manuscript (Theorem~\ref{mainresult}), we solve the inverse problem for function fields over complete discretely valued fields of equal characteristic zero, which are Laurent series fields. While such fields are never algebraically closed, they share with the complex numbers the property of being complete with respect to a metric topology, therefore allowing local calculations involving series. This relates our approach to the classical approach in \cite{tretkoff}, although we use different techniques.  

As a general result that may also be useful in other contexts, we show that in order to solve the inverse problem over finitely generated differential fields over a field of constants $K$, it suffices to solve it over $K(x)$ with derivation $d/dx$ (Corollary~\ref{cor generalresult}). This type of approach has previously been used for groups over algebraically closed fields (e.g.\  for connected groups in \cite{mitschisinger1}) and is sometimes called the {\em Kovacic trick}.

Our approach to solving the inverse problem over rational function fields over Laurent series fields is based on a recent version of patching methods (\cite{HH}) and inspired by the use of patching methods in ordinary Galois theory (see \cite{HarbaterMSRI} for an overview).  In ordinary Galois theory, the patching machinery reduces the realization of a given finite group to the (local) realization of generating (e.g.\ cyclic) subgroups. In the differential setup, there are two additional complications. The first is that in differential Galois theory, it is not possible to patch the local Picard-Vessiot rings (which are the analogs of Galois extensions), since these are not finite as algebras. Instead, we apply patching directly to the differential equations via a factorization property related to patching; this is equivalent to patching the corresponding differential modules. The second complication is that while any finite group is generated by its cyclic subgroups, an analogous statement for linear algebraic groups is true only over algebraically closed fields: Every linear algebraic group over an algebraically closed field is generated by (a finite number of) finite cyclic groups and copies of the multiplicative and additive group. To overcome this issue, we apply the patching method after base change to a finite extension of the field of constants in such a way that the result descends to the original field. 

The advantage of our method is that the building blocks (i.e., the local differential equations for finite cyclic, multiplicative, and additive groups) are very easy to find and verify. 
Because we apply the patching machinery only in a very simple case in which the process can be described explicitly, we do not require the reader to be familiar with \cite{HH}.
We note that a related but somewhat different strategy, to handle a special case of the problem, was sketched by the second author in \cite{Hartmann-OW}.

We expect that the results of this manuscript can be applied to solve the
inverse differential Galois problem for function fields over constant
fields other than Laurent series, for example $p$-adic fields (this is
work in progress).

\medskip

{\bf Organization of the manuscript:} Section~1 lists some basic facts about Picard-Vessiot theory over non-algebraically closed fields of constants. It also contains auxiliary results about the linearization of the differential Galois group, a statement and consequences of the Galois correspondence, and descent results for Picard-Vessiot rings.
Section~2 describes the patching setup for differential equations, proves the main patching result and gives some examples. Section~3 is concerned with the generation of linear algebraic groups by ``simple'' subgroups and with finding differential equations which have those groups as differential Galois groups, the so-called {\em building blocks} for patching. Finally, Section~4 solves the inverse problem over rational function fields by combining the results of the previous two sections. It also contains a result stating that solving the inverse problem over rational function fields implies a solution over arbitrary finitely generated differential fields (over the same field of constants). Combining this with the result obtained for rational function fields gives the main Theorem~\ref{mainresult}.

\medskip

{\bf Acknowledgments:}
The authors wish to thank Tobias Dyckerhoff and Michael Wibmer for helpful discussions. We also thank Michael Singer for comments related to the contents of this manuscript.
\begin{section}{Picard-Vessiot theory}
Throughout this manuscript, all fields are assumed to be of characteristic zero. If $F$ is a field, we write $\bar{F}$ for its algebraic closure. If $R$ is an integral domain, we write $\Frac(R)$ for its field of fractions. If $R$ is a differential ring, we write $C_R$ for its ring of constants. If $F$ is a differential field, then $C_F$ is a field which is algebraically closed in~$F$. 

We first record some facts from the Galois theory of differential fields with arbitrary fields of constants. We refer to \cite{Dyc} for details. 
 
  Let $(F,\del)$ be a differential field with field of constants $C_F=K$. Given a matrix $A\in F^{n\times n}$ and a differential ring extension $R/F$, a \textbf{fundamental solution matrix for $A$} is a matrix $Y \in \GL_n(R)$ which satisfies the differential equation $\del(Y)=AY$. A differential ring extension $R/F$ is called a \textbf{Picard-Vessiot ring} for $A \in F^{n\times n}$ if it satisfies the following conditions: The ring of constants of $R$ is $C_R=K$, there exists a fundamental solution matrix $Y \in \GL_n(R)$, $R$ is generated by the entries of $Y$ and $Y^{-1}$ (we write $R=F[Y,Y^{-1}]$), and $R$ is differentially simple (i.e., has no nontrivial ideals which are stable under $\del$). Differential simplicity implies that a Picard-Vessiot ring is an integral domain. Its fraction field  is called a \textbf{Picard-Vessiot extension}. Note that since $K$ is not assumed to be algebraically closed, Picard-Vessiot rings need not exist (e.g.\ \cite{seidenberg}) and might not be unique. The \textbf{torsor theorem} states that if $R$ is a Picard-Vessiot ring, there is an $R$-linear isomorphism of differential rings $R\otimes_F R\cong R\otimes_K C_{R\otimes_F R}$. 
  
The \textbf{differential Galois group} of a Picard-Vessiot ring $R/F$ is defined as the group functor $\underline{\Aut}^\del(R/F)$ from the category of $K$-algebras to the category of groups that sends a $K$-algebra $S$ to the group $\Aut^\del(R\otimes_K S/F\otimes_K S)$ of differential automorphism of $R\otimes_K S$ that are trivial on $F\otimes_KS$. This functor is represented by the $K$-Hopf algebra $C_{R\otimes_F R}=K[Y\otimes Y^{-1}, Y^{-1}\otimes Y]$. We conclude that the differential Galois group of $R/F$ is an affine group scheme of finite type over $K$ which is necessarily (geometrically) reduced since the characteristic is zero (\cite{Oort}; see also \cite{Cartier}); i.e.\ it is a \textbf{linear algebraic group over $K$} with coordinate ring $K$-isomorphic to $C_{R\otimes_F R}$. Moreover, the torsor theorem asserts that the affine variety $\operatorname{Spec}(R)$ defined by a Picard-Vessiot ring $R$ with Galois group $G$ is a $G$-torsor and  $R\otimes_F \bar{F} \cong  C_{R\otimes_F R}\otimes_K \bar{F}$, which implies $\operatorname{trdeg} (\Frac(R)/F) = \operatorname{dim}(G)$. It is immediate from the definitions that if $K'/K$ is an algebraic \textbf{extension of constants}, then $R\otimes_K K'$ is a Picard-Vessiot ring over $F\otimes_K K'$ whose differential Galois group is the base change of the differential Galois group of $R/F$ from $K$ to $K'$.  We remark that the definition of the differential Galois group requires the use of a Picard-Vessiot ring rather than a Picard-Vessiot extension; for this reason most statements in this manuscript are phrased in terms of a Picard-Vessiot ring.\\[-8pt]
  
When constructing Picard-Vessiot rings, we will frequently use the following well-known criterion (\cite[Corollary 2.7]{Dyc}).

\begin{prop}\label{existencePVR}
Let $L/F$ be an extension of differential fields with constants $C_L=C_F$ and let $A \in F^{n\times n}$. Assume that there exists a fundamental solution matrix $Y \in \GL_n(L)$ for $A$, i.e., $\del(Y)=AY$. Then $R=F[Y,Y^{-1}]\subseteq L$ is a Picard-Vessiot ring for $A$. 
\end{prop}

The next proposition gives a criterion to determine which elements of a Picard-Vessiot extension are contained in a Picard-Vessiot ring. Elements satisfying the condition in the proposition are also called {\em differentially finite}.

\begin{prop}\label{differentiallyfinite}
Let $R/F$ be a Picard-Vessiot ring for some matrix $A$. Then an element $a \in \Frac(R)$ lies in $R$ if and only if $a,\del(a),\del^2(a),\dots$ span a finite dimensional $F$-vector space.  
\end{prop}

\begin{proof}
In the case that $C_F$ is algebraically closed, this is a well-known statement (see \cite[Cor.~1.38]{singervanderput}), which can be applied to $R\otimes_K \bar{K}$ over the field $L:=F \otimes_K{\bar{K}}$.  Here $R\otimes_K \bar{K}$ is a Picard-Vessiot ring since ${\bar K}/K$ is an algebraic extension of constants. Let $a \in R \subseteq R\otimes_K \bar{K}$. Thus $a,\del(a),\del^2(a),\dots$ span a finite dimensional vector space over $L$, so there is an $r \in \N$ such that $a,\del(a),\del^2(a),\dots, \del^r(a)$ are linearly dependent over $L$. Let $V$ be the $F$-vector space spanned by $a,\del(a),\dots,\del^r(a)$. Then 
$$\dim_F(V)=\dim_{L}(V\otimes_F L)=\dim_{L}(V\otimes_K \bar{K})\leq r;$$
and we conclude that $a,\del(a),\dots,\del^r(a)$ are linearly dependent over $F$, giving the forward direction. The proof of the converse direction is the same as in the case $C_F$ is algebraically closed; see \cite[Cor.~1.38, proof of (3)$\Rightarrow$(1)]{singervanderput}.
\end{proof}

In particular, any two Picard-Vessiot rings $R, R'$ with the same field of fractions (but possibly for different matrices) are necessarily equal.

\begin{subsection}{Linearization of the differential Galois group}
The differential Galois group of a Picard-Vessiot ring was defined above as an abstract linear algebraic group. A choice of fundamental solution matrix determines an embedding into a general linear group as follows. 

\begin{prop}\label{thm.galoistheory}
Let $R=F[Y,Y^{-1}]$ be a Picard-Vessiot ring over $F$ with field of constants $C_F=K$. Then there is a closed embedding of linear algebraic groups $$\Psi_Y =\Psi_{Y,K} \colon \underline{\Aut}^\del(R/F)\hookrightarrow \GL_{n,K}$$ such that for all $K$-algebras $S$: $$\Psi_Y(S) \colon \Aut^\del(R\otimes_KS/F\otimes_KS)\hookrightarrow \GL_n(S), \ \sigma \mapsto (Y\otimes 1)^{-1}\sigma(Y\otimes 1).$$ 
\end{prop}

\begin{proof}
Since $\sigma$ is a differential automorphism, $\partial\left((Y \otimes 1)^{-1}\sigma(Y \otimes 1)\right)=0$, i.e., $(Y \otimes 1)^{-1}\sigma(Y \otimes 1)$ has entries in $C_{R\otimes_K S} = K \otimes_K S \cong S$. Therefore, $\Psi_Y(S)$ is a well-defined map. Let $K[Z,Z^{-1}]$ denote the coordinate ring of $\GL_n$ over $K$. Let $\rho \colon \underline{\Aut}^\del(R/F)\hookrightarrow \GL_{n,K}$ be the closed embedding induced by the surjection on the coordinate rings $K[Z,Z^{-1}]\to K[Y\otimes Y^{-1},  Y^{-1}\otimes Y], \ Z \mapsto Y\otimes Y^{-1}$. It is then easy to deduce that $\rho(S)=\Psi_Y(S)$ for every $K$-algebra $S$ by examining the isomorphism $\Aut^\del(R\otimes_K S/F \otimes_K S)\to \Hom_K(K[Y\otimes Y^{-1}, Y^{-1}\otimes Y], S)$. Hence $\Psi_Y=\rho$ is in fact a closed embedding of linear algebraic groups. 
\end{proof}

We define $\Gal^\del_Y(R/F)\leq \GL_n$ as the image of the differential Galois group $\underline{\Aut}^\del(R/F)$ under the linearization $\Psi_Y$ defined in Proposition \ref{thm.galoistheory} and we also call it \textbf{the differential Galois group of $R/F$ (with respect to $Y$}). A different choice of fundamental solution matrix leads to a differential Galois group which is conjugate by an element of $\GL_n(K)$. We will write expressions like $\G\leq \GL_n$ to denote a linear algebraic group with an embedding into a fixed copy of $\GL_n$. 

The following example illustrates the above and will be used in our building blocks in Section~\ref{building blocks subsection}.

\begin{ex}\label{example Ga}
Let $F$ be a differential field with field of constants $K$ and let $a \in F$. We consider the differential equation $\del(y)=Ay$ with $A=\begin{pmatrix}
0 & a \\ 0&0
\end{pmatrix}$. If there exists a differential field extension $L/F$ with $C_L=K$ and an element $y \in L$ with $\del(y)=a$, then $R=F[y]$ is a Picard-Vessiot ring for $A$ over $F$ by Proposition \ref{existencePVR}. Indeed, $Y=\begin{pmatrix}
1 & y \\ 0&1
\end{pmatrix}\in \GL_2(L)$ is a fundamental solution matrix for $A$ and $R=F[Y,Y^{-1}]$. Let $S$ be a $K$-algebra and $\sigma \in \Aut^\del(R\otimes_KS/F\otimes_KS)$. As $$\del(\sigma(y\otimes 1))=\sigma(\del(y\otimes 1))=\sigma(a\otimes 1)=a\otimes 1=\del(y\otimes 1),$$ there exists an $\alpha_\sigma \in C_{R\otimes_K S}=K\otimes_K S\simeq S$ with $\sigma(y\otimes 1)=y\otimes 1+\alpha_\sigma$. Moreover, $Y^{-1}\sigma(Y)=\begin{pmatrix}
1 & \alpha_\sigma \\ 0&1
\end{pmatrix}$, hence $\Gal^\del_Y(R/F)(S)$ is a subgroup of $\Ga(S)$ in its two-dimensional representation and $\Gal^\del_Y(R/F)$ is a subgroup scheme of $\Ga$. (Note that since ${\mathbb G}_a$ has no proper closed subgroups in characteristic zero, it must in fact be either the whole group or trivial.)
\end{ex}

\begin{rem} \label{remark equal groups}
One advantage of working within a fixed $\GL_n$ is the following. Two linear algebraic groups $\G\leq \GL_n$ and $\Hcal \leq \GL_n$ defined over $K$ are equal (as subgroups of $\GL_n$) if and only if $\G(\bar{K})=\Hcal(\bar{K})$ inside $\GL_n(\bar{K})$, since the $\bar{K}$-rational points of a linear algebraic group are dense. In particular, $\G$ and $\Hcal$ are isomorphic over $K$ if $\G(\bar{K})=\Hcal(\bar{K})$; whereas in general, $\G(\bar{K})\cong\Hcal(\bar{K})$ for abstract linear algebraic groups would imply only that $\G$ and $\Hcal$ are isomorphic over $\bar{K}$.  
\end{rem}

The following lemma is immediate from the definitions (in the language of differential modules, the transformation corresponds to a change of basis). 

\begin{lem}\label{modulebasechangelemma}
Let $R/F$ be a Picard-Vessiot ring for a matrix $A$ with fundamental
solution matrix $Y \in \GL_n(R)$. Then for each $B \in \GL_n(F)$, $R/F$ is
also a Picard-Vessiot ring for $\del(B)B^{-1}+BAB^{-1}$ with fundamental
solution matrix $BY \in \GL_n(R)$, and
$\Gal^\del_{Y}(R/F)=\Gal^\del_{BY}(R/F)$ inside $\GL_{n,C_F}$.
\end{lem}

\begin{lem}\label{fieldbasechangelemma}
Let $R/F$ be a Picard-Vessiot ring for a matrix $A$ with fundamental
solution matrix $Y \in \GL_n(R)$. Let $L/F$ be a differential field
extension such that $R$ and $L$ are contained in a common differential
field extension with constants $C_F$. Then the compositum $LR=L[Y,Y^{-1}]$
is a Picard-Vessiot ring for $A$ over $L$, 
and $\Gal^\del_{Y}(LR/L)\leq \Gal^\del_{Y}(R/F)$ as subgroups of $\GL_{n, C_F}$.
\end{lem}

\begin{proof} Abbreviate $K=C_F$. By Proposition \ref{existencePVR},  $LR$
is a Picard-Vessiot ring for $A$ over $L$. For the second assertion, let
$S$ be any $K$-algebra, and let $\sigma
\in\underline{\Aut}^\del(LR/L)(S)=\Aut^\partial(LR\otimes_{K}S/L\otimes_K
S)$.  Then $\sigma$ acts on $LR\otimes_{K}S$ and hence on
$\GL_n(LR\otimes_{K}S)$, taking $Y\otimes 1$ to $(Y\otimes 1)Z$ where $Z
=\Psi_Y(S)(\sigma)\in \Gal^\del_Y(LR/L)(S)\leq \GL_n(S)$ with $\Psi_Y$ as in
Proposition~\ref{thm.galoistheory}. Since $R=F[Y,Y^{-1}]$, $\sigma$
restricts to an automorphism of $R\otimes_K S$. 
The corresponding monomorphism $\Gal^\del_Y(LR/L)(S) \to \Gal^\del_Y(R/F)(S)$ maps $Z$ to $Z$.  This defines an inclusion $\Gal^\del_Y(LR/L) \leq \Gal^\del_Y(R/F)$
of subgroups of $\GL_{n,K}$.
\end{proof}

\noindent If moreover the tensor product of $L$ and $R$ is a compositum
without new constants, we even obtain equality:

\begin{prop}\label{extending PVR}
Let $R/F$ be a Picard-Vessiot ring for a matrix $A$ with fundamental
solution matrix $Y \in \GL_n(R)$. Let $L/F$ be a differential field
extension and assume that $L\otimes_F R$ is an integral domain such that
its field of fractions has constants $C_F$. Then $L\otimes_F R$ is a
Picard-Vessiot ring for $A$ over $L$ with fundamental solution matrix
$1\otimes Y$. Moreover, we have an equality of differential Galois groups
inside $\GL_{n,C_F}$:
\[\Gal^\del_{1\otimes Y}(L\otimes_F R/L)=\Gal^\del_Y(R/F)\]
\end{prop}

\begin{proof} Abbreviate $K=C_F$. View $L\otimes_F R$ as the compositum of
$L$ and $R$ inside $\Frac(L\otimes_F R)$. Then Lemma
\ref{fieldbasechangelemma} implies that $L\otimes_F R=F[1\otimes Y,
(1\otimes Y)^{-1}]$ is a Picard-Vessiot ring over $L$ and
$\Gal^\del_{1\otimes Y}(L\otimes_F R/L)\leq \Gal^\del_Y(R/F)$ as subgroups of
$\GL_{n,K}$. To obtain equality, it suffices to show that for any
$K$-algebra $S$ and any $\sigma_0 \in
\Aut^\partial(R\otimes_{K}S/F\otimes_K S)$ there exists an element $\sigma
\in\Aut^\partial((L\otimes_F R)\otimes_{K}S/L\otimes_K S)=\Aut^\partial(L
\otimes_F (R\otimes_{K}S)/L\otimes_F(F \otimes_K S))$ that restricts to
$\sigma_0$. But $\sigma=\operatorname{id}_L\otimes_F\sigma_0$ has that property, which concludes the proof.
\end{proof}

\noindent In the special case when $R/F$ is a finite Galois extension and $L/F$ is a finite extension, the condition that $L\otimes_F R$ is an integral domain is equivalent to saying that $L$ and $R$ are linearly disjoint over $F$ and the condition that the field $L\otimes_F R$ does not have new constants is equivalent to saying that $L\otimes_F R$ is regular over $K$.  (Recall that a field extension $A/K$ is {\em regular} if $A$ is linearly disjoint from $\bar K$ over $K$.)

\end{subsection}

\begin{subsection}{The Galois correspondence}
Let $R/F$ be a Picard-Vessiot ring together with a fundamental solution matrix  $Y \in \GL_n(R)$. Set $G:=\underline{\Aut}^\del(R/F)$, $\G:=\Gal^\del_Y(R/F)\leq \GL_n$ and $E:=\Frac(R)$, $K:=C_F$. Then for a closed subgroup $\Hcal \leq \G$ (defined over~$K$), consider the inverse image $H\leq G$ of $\Hcal$ under the isomorphism $\Psi_Y: G \to \G$. The \textbf{functorial invariants} $E^H$ in $E$ are defined as the elements $a/b \in \Frac(R)$ satisfying 
\[ 
  \sigma(a\otimes 1)(b\otimes 1)=(a\otimes 1) \sigma(b\otimes 1)  \] for all $K$-algebras $S$ and all $\sigma \in H(S)\subseteq \Aut^\del(R\otimes_K S/F\otimes_K S)$. As $H(\bar{K})$ is dense in $H$, it even suffices to check the above condition for $S=\bar{K}$. The \textbf{Galois correspondence} (\cite[Thm.~4.4]{Dyc}) then states that there is an inclusion-reversing bijection between closed subgroups of $G$ (defined over $K$) and differential subfields $F \subseteq L \subseteq E$, given by \[ H\mapsto E^H \text{ and } L\mapsto \underline{\Aut}^\del(R/L).  \]

We illustrate with an example why we need to consider functorial invariants rather than invariants on $K$-points.

\begin{ex}\label{example-invariants}
Consider $F=\Q(x)$ with derivation $d/dx$ and $R=F[\sqrt[3]{x}]$. Then $R=\Frac(R)$ is a Picard-Vessiot ring over $F$ for the differential equation $\del(y)=\frac{1}{3x}y$ with fundamental solution matrix $y=\sqrt[3]{x}$ and differential Galois group $G\cong \mu_3$, the subgroup of $\GL_3$ defined by the equation $X^3-1=0$. Note that $\mu_3(\Q)=\{1 \}$; hence $\Aut^\partial(R/F)=\Aut(R/F)=\{\operatorname{id} \}$, and $R^{\Aut(R/F)}$ strictly contains $F$.  But $R^G = F$.
\end{ex}

\medskip

It is known (\cite[Proposition 4.3]{Dyc}) that if $H\leq G$ is a closed normal subgroup, then $E^H$  is the fraction field of a Picard-Vessiot ring $R_H$ over $F$
with $\underline{\Aut}^\del(R_H/F)\cong G/H$. For $H$ equal to the identity component of $G$, we obtain the following.

\begin{lem}\label{closed}
Let $F$ be a differential field and let $R/F$ be a Picard-Vessiot ring with differential Galois group $G$. Write $G^0 $ for the identity component of $G$ and $E=\Frac(R)$. Then $E^{G^0 }$ is the algebraic closure of $F$ in $E$, and it is a finite field extension of degree $[E^{G^0 }:F]=|G(\bar{K})/G^0 (\bar{K})|$.
\end{lem}

\begin{proof}
Since $\trd(E/E^{G^0 })=\dim(G^0 )=\dim(G)=\trd(E/F)$ and $E$ is finitely generated over $F$, the extension $E^{G^0 }/F$ is finite. On the other hand, every algebraic subextension $F\subseteq L \subseteq E$ is a differential extension, and thus $H:=\underline{\Aut}^\del(R/L)$ is a closed subgroup of $G$ of the same dimension. Therefore, $G^0  \subseteq H$ and $L=E^H \subseteq E^{G^0 }$. Hence $E^{G^0 }$ is the algebraic closure of $F$ in $E$. While the finite field extension $E^{G^0}/F$ might not be Galois, the compositum $E^{G^0}\bar{K}\cong E^{G^0}\otimes_K {\bar K}$ is a finite extension of $F\bar{K}\cong F\otimes_K {\bar K}$ of the same degree as $E^{G^0}/F$. Moreover, $\Aut(E^{G^0}\bar{K}/F\bar{K})\cong (G/G^0)(\bar{K})\cong G(\bar{K})/G^0 (\bar{K})$ and $(E^{G^0}\bar{K})^{\Aut(E^{G^0}\bar{K}/F\bar{K})}=F\bar{K}$ (because $\bar{K}$ is algebraically closed).  Hence $E^{G^0}\bar{K}/F\bar{K}$ is a finite Galois extension of degree $|G(\bar{K})/G^0 (\bar{K})|$ and the claim follows. 
\end{proof}

\end{subsection}

\begin{subsection}{Galois descent for Picard-Vessiot rings}
 
 Let $F_0$ be a differential field with field of constants $K_0$. Given a linear algebraic group $\G\leq \GL_n$ over $K_0$, it might be easier to realize $\G_K$ as a differential Galois group over $F_0K$ for some finite extension of constants $K/K_0$. We assume that $K/K_0$ is a finite Galois extension with Galois group $\Gamma$. As $K_0$ is algebraically closed in $F_0$, $F_0K \cong F_0\otimes_{K_0}K$ is Galois over $F_0$ with group $\Gamma$ (and the action of $\Gamma$ commutes with the derivation). In our applications, we construct a Picard-Vessiot ring over $F_0K$ such that the fundamental solution matrix is $\Gamma$-invariant. The following lemma explains that this Picard-Vessiot ring then descends to a Picard-Vessiot ring over $F_0$ with differential Galois group $\G$.
 
 \begin{lem}\label{lemmainvariantPVR}
 Let $K/K_0$ be a finite Galois extension with Galois group $\Gamma$. Let $F_0$ be a differential field with $C_{F_0}=K_0$ and set $F=F_0K$.  Further, let $L/F$ be an extension of differential fields with $C_{L}=C_F=K$ and such that the action of~$\Gamma$ on~$F$ extends to an action on $L$ via differential automorphisms. If $R=F[Y,Y^{-1}]\subseteq L$ is a Picard-Vessiot ring over $F$ such that $Y \in \GL_n(L)$ is invariant under the action of $\Gamma$, then $R_0:=F_0[Y,Y^{-1}]$ is a Picard-Vessiot ring over $F_0$ with $\Gal^\del_{Y}(R_0/F_0)_K=\Gal^\del_Y(R/F)$ as subgroups of $\GL_n$. In particular, if $\Gal^\del_Y(R/F)=\G_K$ for a linear algebraic group $\G \leq \GL_{n,K_0}$, then $\Gal^\del_{Y}(R_0/F_0)=\G$.
 \end{lem}
 
\begin{proof}
 Note that $A:=\del(Y)Y^{-1} \in F^{n\times n}$ is $\Gamma$-invariant, hence $A \in F_0^{n\times n}$. The field $L^\Gamma$ of $\Gamma$-invariants in $L$ is a differential field extension of $F_0$ with $C_{L^\Gamma}=K_0=C_{F_0}$, since $C_{L}=K$. As $Y$ is contained in $\GL_n(L^\Gamma)$, Proposition \ref{existencePVR} implies that $R_0$ is a Picard-Vessiot ring for $A$ over $F_0$. Let $\Hcal \leq \GL_n$ be the linear algebraic group $\Gal^\del_Y(R_0/F_0)$ defined over $K_0$.  Since $K_0$ is algebraically closed in $\Frac(R_0)$, the natural map $R_0\otimes_{K_0}K \to R$ is an isomorphism, and the induced map on matrices sends $Y\otimes_{K_0} 1$ to $Y$.  Hence the natural map $\iota:{\Aut}^\del(R\otimes_K \bar{K}/F\otimes_K\bar{K})
\to {\Aut}^\del(R_0\otimes_{K_0} \bar{K}/F_0\otimes_{K_0}\bar{K})$ is an isomorphism.  
It is given by $\sigma \mapsto \phi^{-1}\sigma\phi$, where the 
differential $\bar{K}$-isomorphism
$\phi$ is the composition 
$$R_0\otimes_{K_0}\bar{K} \to
{R_0\otimes_{K_0} K \otimes_K \bar{K}} \to
R\otimes_K \bar{K},$$ whose induced map on matrices sends
$Y\otimes_{K_0}1$ to $Y\otimes_K 1$. 

We claim that the isomorphism $\iota$ yields the equality
$$\Hcal(\bar{K})=\Gal^\del_Y(R_0/F_0)(\bar{K})=\Gal^\del_Y(R/F)(\bar{K})$$
as subsets of $\GL_n(\bar{K})$.  Namely,
for any $\sigma \in {\Aut}^\del(R\otimes_{K} \bar{K}/F\otimes_{K}\bar{K})$,
 \begin{eqnarray*}
  \Psi_{Y,K}(\bar K)(\sigma) &=&  (Y\otimes_{K}1)^{-1}\sigma(Y\otimes_{K}1)
\\
  &=& \phi((Y\otimes_{K_0}1)^{-1}\iota(\sigma)(Y\otimes_{K_0}1))
\\
  &=& (Y\otimes_{K_0}1)^{-1}\iota(\sigma)(Y\otimes_{K_0}1)
\\
  &=& \Psi_{Y,K_0}(\bar K)(\iota(\sigma)).
\end{eqnarray*}
(The third equality holds since all entries of the matrix are contained in $\bar{K}$.) The claim follows.
 
We conclude that $\Hcal_K=\Gal^\del_Y(R/F)$ as subgroups of $\GL_n$ as asserted. In particular, if $\Gal^\del_Y(R/F)=\G_K$ 
for some $\G \le \GL_{n,K_0}$, then $\Hcal(\bar{K})=\G_K(\bar{K})=\G(\bar{K})$ as subgroups of $ \GL_{n}(\bar{K})$, and hence $\Hcal=\G$. 
\end{proof}
\end{subsection}
\end{section}

\begin{section}{Application of patching to differential equations}\label{sectionpatching}

Throughout this section, let $K=k((t))$ for some field $k$ of characteristic zero; let $F=K(x)$; and consider a derivation $\del$ on $F$  with $C_F=K$. Note that $\del=\del(x)\cdot \del/\del x$. Moreover, $F$ is the function field of the projective $x$-line $\mathbb{P}^1_{k[[t]]}$ over the discrete valuation ring $k[[t]]$. In \cite{HH}, a collection of field extensions $F_P$, $F_{\wp}$ and $F_U$ were considered. Whereas the definitions in loc.~cit.\ are geometric, we only need a special case, in which the description is very explicit:

If $P \in \mathbb{A}^1_k \subset \mathbb{P}^1_k$ is a rational point defined by $x=b$ for some $b \in k$, then we consider the fields
\begin{eqnarray*}
F_P&=&k((x-b,t)) \quad \text{ and}\\
F_{\wp(P)}&=&k((x-b))((t)),
\end{eqnarray*}
where $k((x-b,t))$ denotes the fraction field of the power series ring $k[[x-b,t]]$ in two variables. Given a non-empty finite set $\Pcal$ of $k$-rational points, we have an index set $\Bcal=\{\wp(P)|\; P \in \Pcal\}$ in bijection with $\Pcal$.  If $\Pcal$ consists of points $z=b_i$ for some $b_i \in k$ and $1\leq i \leq m$, then we let $U$ be the complement of $\Pcal$ in $\mathbb{P}^1_k$, and we write
\[F_U=\Frac(k[(x-b_1)^{-1},\dots,(x-b_m)^{-1}][[t]]). \]

As explained in \cite{HH}, there are inclusions $F \subseteq F_U, \! F_P \subseteq F_{\wp(P)}$ for all $P \in \Pcal$. These are in fact inclusions of differential fields, where we equip $F_U$, $F_P$, and $F_{\wp(P)}$  with the derivation $\del(x)\cdot \del/\del x$; moreover $C_{F_U}=C_{F_P}=C_{F_{\wp(P)}}=C_F=K$ for all $P \in \Pcal$. In particular, if $A \in F^{n\times n}$ is such that there exists a fundamental solution matrix $Y \in \GL_n(F_U)$, then $F[Y,Y^{-1}]$ is a Picard-Vessiot ring for $A$ over $F$ (by Proposition \ref{existencePVR}).

The method of patching over the fields $(F,F_U,F_P,F_\wp)$ relies on the following two properties. 

\begin{thm}\label{patching}$ $ \vspace{-0.2cm}
\begin{enumerate}[(a)]
 \item\label{factorization} \underline{Simultaneous factorization property:} Let $n \in \N$. If $(Y_\wp)_{\wp \in \Bcal}$ is a collection of matrices $Y_\wp \in \GL_n(F_\wp)$ then there exist matrices $B_P \in \GL_n(F_P)$ for each $P \in \Pcal$ and one matrix $Y \in \GL_n(F_U)$ such that for each $P \in \Pcal$, $Y_{\wp(P)}=B_P^{-1}Y$ in $\GL_n(F_{\wp(P)})$. 
 \item\label{intersection} \underline{Intersection property:} If $x \in F_U$ is such that for each $P \in \Pcal$, $x$ is contained in $F_P$ when considered as an element of $F_{\wp(P)}$, then $x \in F$. 
\end{enumerate}
\end{thm}

For a proof of the simultaneous factorization property, see \cite[Thm 5.10]{HH} and \cite[Prop. 2.2]{HHKtorsor}. The intersection property is stated in \cite[Prop. 6.3]{HH}.

\begin{Def}\label{actonpatchingdata}
In the above setup, an {\em action} of a finite group $\Gamma$ on the {\em differential patching data} $(\Pcal,\Bcal,U)$ consists of the following:
\begin{enumerate}[(1)]
 \item a left action of $\Gamma$ on $F$ via differential automorphisms.
\item a left action of $\Gamma$ on $F_U$ via differential automorphisms, extending the action of $\Gamma$ on $F$.
\item a right action of $\Gamma$ on the finite set $\Pcal$. 
\item for each $\sigma \in \Gamma$ and each $P \in \Pcal$, an isomorphism $\sigma \colon F_{P^\sigma} \to F_{P}$ of differential fields extending $\sigma \colon F \to F$ such that for all $\sigma,\tau \in \Gamma$, $\sigma\tau \colon F_{P^{\sigma\tau}} \to F_{P}$ is the composition $\sigma \circ \tau \colon F_{P^{\sigma\tau}}\to F_{P^\sigma} \to F_{P}$.
\item for each $\sigma \in \Gamma$ and each $P \in \Pcal$, an isomorphism of differential fields $\sigma \colon F_{\wp(P^\sigma)}\to F_{\wp(P)}$ extending both $\sigma\colon F_U \to F_U$ and $\sigma \colon F_{P^\sigma}\to F_{P}$ such that for all $\sigma,\tau \in \Gamma$, $\sigma\tau \colon F_{\wp(P^{\sigma\tau})} \to F_{\wp(P)}$ is the composition $\sigma \circ \tau \colon F_{\wp(P^{\sigma\tau})}\to F_{\wp(P^{\sigma})} \to F_{\wp(P)}$.
\end{enumerate}
\end{Def}

\begin{ex}\label{example action}
(a)  Let $k_0\leq k$ such that $k/k_0$ is a finite Galois extension, let $e\ge 1$ be a natural number such that $k$ contains a primitive $e$-th root of unity, and set $t_0=t^e$. Then $K=k((t))$ is a finite Galois extension of $K_0=k_0((t_0))$, and $\Gamma:=\operatorname{Gal}(K/K_0)$ is the semi-direct product of the cyclic group of order $e$ and the group $\Galf(k/k_0)$.  In particular, $\Gamma$ surjects onto $\operatorname{Gal}(k/k_0)$.
Note that $F=K(x)$ is a finite Galois extension of $F_0=K_0(x)$ whose Galois group is isomorphic to $\Gamma$ and acts on $F$ as a differential field. The action of $\Gamma$ on $F$ over $F_0$ (from the left) induces an action of $\Gamma$ on the $x$-line $\mathbb{P}^1_{k[[t]]}$ over $\mathbb{P}^1_{k_0[[t_0]]}$ from the right. In particular, there is an induced action of $\Gamma$ on $\mathbb{P}^1_{k}$ over $\mathbb{P}^1_{k_0}$ from the right. If $P \in \mathbb{P}^1_{k}$ is a point of the form $x=b$ for some $b \in k$, then $P^\sigma$ is defined by $x=\sigma^{-1}(b)$. The induced isomorphism \[\sigma\colon F_{P^\sigma}=k((x-\sigma^{-1}(b),t)) \to F_{P}=k((x-b,t))\] is a differential isomorphism, as is the induced isomorphism \[\sigma \colon F_{\wp(P^\sigma)}=k((x-\sigma^{-1}(b)))((t)) \to F_{\wp(P)}=k((x-b))((t)).\]  Notice that if $\Pcal \subseteq \mathbb{P}^1_{k}$ is a finite set of closed points invariant under the action of $\Gamma$, then the action of $\Gamma$ on $\mathbb{P}^1_{k[[t]]}$ also induces an action of $\Gamma$ on $F_U$ as a differential field. Therefore, $\Gamma$ acts on the differential patching data $(\Pcal, \Bcal, U)$.

(b)  Observe that the $\Gamma$-orbit of a closed point $x=b$ as above consists of at most $|\Galf(k/k_0)|\leq|\Gamma|$ elements. To obtain an action with orbits of full length $|\Gamma|$ 
for use in Section~\ref{section inverse}, we will use the following construction. Suppose again that $k$ contains a primitive $e$-th root of unity $\zeta$. Let $z=x/t$. We will work with the $z$-line $\mathbb{P}^1_{k[[t]]}$ instead of the $x$-line. (In other words, we perform a blow-up at the origin, and then blow down the original component.) Let $\sigma \in \Gamma$. As $t$ is an $e$-th root of $t_0 \in K_0$, $\sigma(t)=\zeta^{n_\sigma}t$ for some $n_\sigma \in \N$, and hence $\sigma(z)=\zeta^{-n_\sigma}z$. Note that the induced action of $\Gamma$ on the $z$-line
$\mathbb{P}^1_k$ maps a point $P$ of the form $z=b$ to the point $P^\sigma$ defined by $z=\zeta^{n_{\sigma^{-1}}}\sigma^{-1}(b)$; and the induced isomorphisms $\sigma\colon F_{P^\sigma}=k((z-\zeta^{n_{\sigma^{-1}}}\sigma^{-1}(b),t)) \to F_{P}=k((z-b,t))$ and $F_{\wp(P^\sigma)}=k((z-\zeta^{n_{\sigma^{-1}}}\sigma^{-1}(b)))((t)) \to F_{\wp(P)}=k((z-b))((t))$ are again isomorphisms of differential fields (they map $z-\zeta^{n_{\sigma^{-1}}}\sigma^{-1}(b)$ to $\zeta^{-n_\sigma}(z-b)$). The $\Gamma$-orbit of such a point $z=b$ then consists of all points of the form $z=\zeta^{n_\sigma}\sigma(b)$ for $\sigma \in \Gamma$. We will show later that there exist elements $b \in k$ such that this orbit consists of $|\Gamma|$ points.

\end{ex}

\begin{thm}\label{diffpatching} Let $n \in \N$. 
\begin{enumerate}
 \item\label{diffpatching1} Assume that for each $P \in \Pcal$, a matrix $A_P \in F_P^{n\times n}$ is given together with a fundamental solution matrix $Y_P \in \GL_n(F_{\wp(P)})$. Let $\G_P=\Gal^\del_{Y_P}(R_P/F_P) \leq \GL_{n,K}$ be the differential Galois group  of the Picard-Vessiot ring $R_P=F_P[Y_P,Y_P^{-1}]$ for $A_P$ (see Proposition \ref{existencePVR}). Then there exists a matrix $A \in F^{n\times n}$ and a fundamental solution matrix $Y\in \GL_n(F_U)$ for $A$ such that the Picard-Vessiot ring $R=F[Y,Y^{-1}]$ over $F$ has the following property: Its differential Galois group equals the Zariski closure of the group generated by the various subgroups $\G_P$ of $\GL_{n,K}$; i.e.\ $\Gal^\del_Y(R/F)=\bar{<\G_P \ | \ P \in \Pcal >}$.
\item \label{diffpatching2} Assume that moreover a finite group $\Gamma$ acts on the differential patching data $(\Pcal,\Bcal,U)$ and assume that 
$\sigma(Y_{P^\sigma})=Y_P$ in $\GL_n(F_{\wp(P)})$ for each $P \in \Pcal$ and each $\sigma \in \Gamma$. Then the fundamental solution matrix $Y \in \GL_n(F_U)$ can be chosen such that its entries are $\Gamma$-invariant.
\end{enumerate}
\end{thm}

\begin{proof}
Simultaneous factorization (Theorem \ref{patching}.\ref{factorization}) implies that there exists a matrix $Y \in \GL_n(F_U)$ and matrices $B_P \in \GL_n(F_P)$ for each $P \in \Pcal$ such that $Y_P=B_P^{-1}\cdot Y$ when viewed inside $\GL_n(F_{\wp(P)})$ for each $P \in \Pcal$. We set $A=\del(Y)Y^{-1} \in F_U^{n\times n}$. For each $P \in \Pcal$, we compute inside $F_{\wp(P)}^{n\times n}$, viewing $F_U$ and $F_P$ as subfields of $F_{\wp(P)}$:
\begin{eqnarray*}
A&=&\del(Y)Y^{-1} \\
&=&\del(B_PY_P)Y_P^{-1}B_P^{-1}\\
&=&\del(B_P)B_P^{-1}+B_PA_PB_P^{-1} \in F_P^{n\times n},
\end{eqnarray*}
using that $Y_P$ is a fundamental solution matrix for $A_P$.
The intersection property (Theorem \ref{patching}.\ref{intersection}) implies that all entries of $A$ are contained in $F$. By Proposition \ref{existencePVR}, $R=F[Y,Y^{-1}]\subseteq{F_U}$ is a Picard-Vessiot ring for $A\in F^{n\times n}$ over $F$. Set $\G=\Gal^\del_Y(R/F) \leq \GL_{n,K}$. 

We first show that $\G_P \leq \G$ for all $P \in \Pcal$. Note that $R_P=F_P[Y_P,Y_P^{-1}]=F_P[Y,Y^{-1}]=F_PR$ (the compositum is taken inside $F_{\wp(P)}$) for all $P \in \Pcal$ since $Y_P=B_P^{-1}\cdot Y$. Hence 
\[\G_P=\Gal^\del_{Y_P}(R_P/F_P)=\Gal^\del_{Y}(R_P/F_P)\leq\Gal^\del_{Y}(R/F)=\G\] 
by Lemma \ref{modulebasechangelemma} and Lemma \ref{fieldbasechangelemma}. We now let $\Hcal\leq \G$ be the Zariski closure of $<\G_P \ | \ P \in \Pcal >$ in $\GL_n$. We claim that $\G=\Hcal$. Let $H\leq \underline{\Aut}^\del(R/F)$ be the inverse image of $\Hcal \le \G$ under the isomorphism
 $\Psi_Y:\Autd(R/F) \to \G$ (Proposition~\ref{thm.galoistheory}).
  By the Galois correspondence, it suffices to show that $E^H=F$, where $E=\Frac(R)$. Suppose there exists an element $a/b \in E^H\smallsetminus F$ for some $a,b \in R$. Note that $a/b \in E\subseteq F_U\subseteq F_\wp$ for all $\wp \in \Bcal$. The intersection property implies that there exists a $P \in \Pcal$ with $a/b \in F_{\wp(P)} \smallsetminus F_P$. On the other hand, $\G_P \leq \Hcal$, hence $\underline{\Aut}^\del(R_P/F_P) \leq H$.
  Indeed, $\underline{\Aut}^\del(R_P/F_P)$ can be identified with a subgroup scheme of $\underline{\Aut}^\del(R/F)$ via restriction and $\Psi_Y:\Autd(R/F) \to \G$ maps this subgroup scheme to $\Gal^\del_{Y}(R_P/F_P)=\G_P\le \Hcal=\Psi_Y(H)$.
   We conclude that $a/b$ is an element of $\Frac(R)\subseteq \Frac(R_P)$ that is invariant under $\underline{\Aut}^\del(R_P/F_P)$. The Galois correspondence applied to the Picard-Vessiot ring $R_P/F_P$ implies that $a/b$ is contained in $F_P$, a contradiction. This proves part~\ref{diffpatching1}. \\
   
To prove part~\ref{diffpatching2}, it suffices to show that there exists a $B \in \GL_n(F)$ such that $B^{-1}Y \in \GL_n(F_U^\Gamma)$, since then $\Gal^\del_Y(R/F)=\Gal^\del_{B^{-1}Y}(R/F)$ by Lemma \ref{modulebasechangelemma}.

We first claim that for each $\sigma \in \Gamma$, $Y\sigma(Y)^{-1} \in \GL_n(F_U)$ has entries in~$F$. Let $\sigma \in \Gamma$. By the intersection property, it suffices to show that $Y\sigma(Y)^{-1} \in \GL_n(F_P)$  when viewed as an element in $\GL_n(F_{\wp(P)})$ for each $P \in \Pcal$. Let $P \in \Pcal$ and set $Q=P^\sigma \in \Pcal$. By assumption, there is a differential isomorphism $\sigma \colon F_{\wp(Q)} \to F_{\wp(P)}$ restricting to $\sigma \colon F_Q \to F_P$ and restricting to $\sigma \colon F_U \to F_U$. In $\GL_n(F_{\wp(Q)})$, we have $Y=B_QY_Q$, with notation as
in the proof of part~\ref{diffpatching1}; 
hence $\sigma(Y)=\sigma(B_Q)\sigma(Y_Q)=\sigma(B_Q)\sigma(Y_{P^\sigma})=\sigma(B_Q)Y_P$ in $\GL_n(F_{\wp(P)})$. On the other hand, $Y=B_P Y_P$ and we compute inside $\GL_n(F_{\wp(P)})$: $Y\sigma(Y)^{-1}=B_PY_P(\sigma(B_Q)Y_P)^{-1}=B_P\sigma(B_Q)^{-1} \in \GL_n(F_P)$, proving the claim.

Therefore, we obtain a 1-cocycle $\chi \colon \Gamma \to \GL_n(F), \sigma \mapsto Y\sigma(Y)^{-1}$. By Hilbert's Theorem~90, $H^1(\Gamma, \GL_n(F))$ is trivial, hence there exists a $B \in \GL_n(F)$ such that for all $\sigma \in \Gamma$: $Y\sigma(Y)^{-1}=B\sigma(B)^{-1}$. This implies that $B^{-1}Y \in \GL_n(F_U)$ is $\Gamma$-invariant as we wanted to show. 
\end{proof}

Note that part~(b) of the above theorem enables us to use Lemma~\ref{lemmainvariantPVR} and hence to obtain Galois descent.

\begin{ex} $ $ \\  
(1) In this example, we apply Theorem \ref{diffpatching} to show that $\SL_2$ is  a differential Galois group over $F=\R((t))(x)$ (with derivation $\del=\del/\del x$). Let $\G_1,\G_2 \leq \SL_2$ be the subgroups of upper and lower unitary triangular matrices in $\SL_2$. Let $P_1$, $P_2$ be the closed points $x=1$ and $x=2$ on $\mathbb{P}^1_\R$ and consider the patching data $(\Pcal, \Bcal, U)$ induced by $\Pcal=\{P_1,P_2\}$. Then for $j=1,2$, $F_{P_j}=\R((x-j,t))$ and $F_{\wp(P_j)}=\R((x-j))((t))$; and $F_U=\Frac(\R[(x-1)^{-1},(x-2)^{-1}][[t]])$. Set $a_j=-t(x-j)^{-1}(x-j-t)^{-1} \in F \subseteq F_{P_j}$ for $j=1,2$.  Also write $A_{P_1}=\begin{pmatrix} 0& a_1 \\ 0 &0 \end{pmatrix}$,  $A_{P_2}=\begin{pmatrix} 0& 0 \\ a_2 &0 \end{pmatrix}$  and $Y_{P_1}=\begin{pmatrix} 1& y_1 \\ 0 &1 \end{pmatrix}$, $Y_{P_2}=\begin{pmatrix} 1& 0 \\ y_2 &1 \end{pmatrix}$, with $y_j=\sum\limits_{r=1}^\infty \frac{1}{r(x-j)^r}t^r \in F_U \subseteq F_{\wp(P_j)}$ for $j=1,2$. Then $\del(Y_P)=A_PY_P$ for both points $P \in \Pcal$, and $R_P=F_P[Y_P,Y_P^{-1}]$ is a Picard-Vessiot ring for $A_P$ over $F_P$. It is easy to see that $y_j$ is transcendental over $F_{P_j}$ and to deduce that $\Gal^\del_{Y_{P_j}}(R_{P_j}/F_{P_j})=\G_j$ in $\GL_2$ for $j=1,2$ (use Example~\ref{example Ga}). Theorem~\ref{diffpatching}\ref{diffpatching1} now implies that there exists an $A \in F^{2\times 2}$ and a $Y \in \GL_2(F_U)$ such that the Picard-Vessiot ring $R=F[Y,Y^{-1}]$ for $A$ over $F$ has differential Galois group $\Gal^\del_Y(R/F)=\bar{\left< \G_1, \G_2 \right>}=\SL_2$. \medskip

(2) As another example, let us consider $\SO_2$ in its two-dimensional representation 
$$\left\lbrace \begin{pmatrix}
  a & b \\ -b & a   \end{pmatrix} \ | \ a^2+b^2=1 \right\rbrace $$ over $\R$. We realize $\SO_2$ as a differential Galois group over $\R((t))(x)$ by using patching over $F=\C((t))(x)$ (with derivation $\del=\del/\del x$ on both fields). Consider the isomorphism over $\C$ 
 \[\psi \colon \mathbb{G}_m \to \SO_2, \ \lambda \mapsto \frac{1}{2} \begin{pmatrix} \lambda +\lambda^{-1} & i(-\lambda+\lambda^{-1}) \\
               i(\lambda-\lambda^{-1}) & \lambda +\lambda^{-1} \end{pmatrix}.
  \]  Let $P_1$, $P_2$ be the closed points $x=i$ and $x=-i$ on $\mathbb{P}^1_\C$ and consider the patching data $(\Pcal, \Bcal, U)$ induced by $\Pcal=\{P_1,P_2\}$. Then $F_{P}=\C((x\pm i,t))$ and $F_{\wp(P)}=\C((x\pm i))((t))$ for $P=P_1, P_2$, respectively, and $F_U=\Frac(\C[(x-i)^{-1},(x+i)^{-1}][[t]])$. Note that $\Gamma=\operatorname{Gal}(\C/\R)$ acts on $F$ and on $F_U$ via differential automorphisms and the non-trivial element $\sigma \in \Gamma$ induces $\del$-isomorphisms $F_{P_1}\to F_{P_2}$ and $F_{\wp(P_1)}\to F_{\wp(P_2)}$. Therefore, $\Gamma$ acts on the differential patching data $(\Pcal, \Bcal, U)$. Set $y=e^{\frac{t}{x-i}} \in F_U$ and $Y_{P_1}=\psi(y)\in \GL_2(F_U)\leq\GL_2(F_{\wp(P_1)})$. Note that $R_{P_1}=F_{P_1}[Y_{P_1},Y_{P_1}^{-1}]=F_{P_1}[y,y^{-1}]$ is a Picard-Vessiot ring over $F_{P_1}$ for the one-dimensional equation $\del(y)=\frac{-t}{(x-i)^2}y$. It can be shown that $A_{P_1}:=\del(Y_{P_1})Y_{P_1}^{-1}\in R_{P_1}^{2\times 2}$ is contained in $F_{P_1}^{2\times 2}$. Thus $R_{P_1}\subseteq F_{\wp(P_1)}$ is also a Picard-Vessiot ring for $A_{P_1}$ over $F_{P_1}$. Set $Y_{P_2}=\sigma(Y_{P_1})$. Then $A_{P_2}:=\del(Y_{P_2})Y_{P_2}^{-1}=\sigma(A_{P_1}) \in F_{P_2}^{2\times 2}$, and $R_{P_2}=F_{P_2}[Y_{P_2},Y_{P_2}^{-1}]$ is a Picard-Vessiot ring for $A_{P_2}$ over $F_{P_2}$. It is easy to see that $y$ is transcendental over $F_{P_1}$ hence $\Gal^\del_{y}(R_{P_1}/F_{P_1})={\mathbb G}_{m,\C}$ and thus $\Gal^\del_{Y_{P_1}}(R_{P_1}/F_{P_1})=\psi({\mathbb G}_{m,\C})=(\SO_2)_\C$. Similarly, $\Gal^\del_{Y_{P_2}}(R_{P_2}/F_{P_2})=(\SO_2)_\C$. Theorem~\ref{diffpatching}\ref{diffpatching1} implies that there exists an $A \in F^{2\times 2}$ with fundamental solution matrix $Y \in \GL_2(F_U)$ and Picard-Vessiot ring $R=F[Y,Y^{-1}]$ such that $\Gal^\del_Y(R/\C((t))(x))=(\SO_2)_\C$. By Theorem~\ref{diffpatching}\ref{diffpatching2}, we can moreover assume that $Y$ is $\Gamma$-invariant. Lemma~\ref{lemmainvariantPVR} then implies that $R_0:=\R((t))(x)[Y,Y^{-1}]$ is a Picard-Vessiot ring over $\R((t))(x)$ with $\Gal^\del_Y(R_0/\R((t))(x))=\SO_2$.  \medskip
  
(3) Let $T\leq \GL_n$ be a one-dimensional torus defined over $\Q((t))$ that splits over the finite extension $\Q((t))(\sqrt{t})=\Q((s))$ with $s:=\sqrt{t}$, i.e., there is an isomorphism $\psi \colon \mathbb{G}_m \to T$ defined over $\Q((s))$. Set $K=\Q((s))$ and $F=K(x)$ with derivation $\del=\del/\del x$. Then $K/\Q((t))$ is Galois with  Galois group $\Gamma \cong C_2$ and $F/\Q((t))(x)$ is also Galois with Galois group $\Gamma$. Let $\sigma \in \Gamma$ be the non-trivial automorphism. Consider the change of variables $z=x/s \in F$. Then $F=K(z)$ with $\del(z)=1/s$ and $\sigma(z)=-z$. Let $P_1$, $P_2$ be the closed points $z=1$ and $z=-1$ on the $z$-line $\mathbb{P}^1_\Q$ and consider the patching data $(\Pcal, \Bcal, U)$ induced by $\Pcal=\{P_1,P_2\}$. Then $F_{P}=\Q((z\pm 1,s))$ and $F_{\wp(P)}=\Q((z\pm 1))((s))$ for $P=P_1, P_2$, respectively, and $F_U=\Frac(\Q[(z-1)^{-1},(z+1)^{-1}][[s]])$. As explained in Example \ref{example action}(b), $\Gamma$ acts on $F$ and on $F_U$ via differential automorphisms and $\sigma$ induces $\del$-isomorphisms $F_{P_1}\to F_{P_2}$ and $F_{\wp(P_1)}\to F_{\wp(P_2)}$. Therefore, $\Gamma$ acts on the differential patching data $(\Pcal, \Bcal, U)$ permuting $P_1$ and $P_2$. We set $y=e^{\frac{s}{z-1}} \in F_{\wp(P_1)}$ which is transcendental over $F_{P_1}$. Note that $\del(y)y^{-1}= \frac{1}{s}\cdot \frac{-s}{(z-1)^2}\in F_{P_1}$, since $\del(z)=1/s$. We define $Y_{P_1}=\psi(y) \in T(F_{\wp (P_1)})$ and $Y_{P_2}=\sigma(Y_{P_1}) \in T(F_{\wp (P_2)})$ and proceed as in the previous example. Eventually, we
obtain an $A \in F^{n\times n}$ with fundamental solution matrix $Y \in \GL_n(F_U)$ and Picard-Vessiot ring $R=F[Y,Y^{-1}]$ such that $\Gal^\del_Y(R/\Q((s))(x))=T_K$ and moreover $R_0=\Q((t))(x)[Y,Y^{-1}]$ is a Picard-Vessiot ring over $\Q((t))(x)$ with $\Gal^\del_Y(R_0/\Q((t))(x))=T$. 

\end{ex}

\end{section}

\begin{section}{Constructing extensions}

\subsection{Linearizations}
In this subsection, $K$ denotes a field of characteristic zero.
We begin by showing that every linear algebraic group admits a finite set of ``simple'' generating subgroups, after passing to a finite field extension.

\begin{prop}\label{groupsplits}
Let $\G$ be a linear algebraic group defined over $K$. Then there exists a finite extension $L/K$ and finitely many closed subgroups $\G_1,\dots,\G_r \leq \G_L$ such that $\G_L$ is generated by $\G_1,\dots, \G_m$ and such that each $\G_i$ is isomorphic (over $L$) to either $\mathbb{G}_a$ or $\mathbb{G}_m$ or a finite (constant) cyclic group.
\end{prop} 

\begin{proof}
It suffices to show this for $L$ replaced by an algebraic closure $\bar K$ of $K$. By the theorem of Borel-Serre \cite[Lemme 5.11]{BorelSerre}, $\G_{\bar K}$ is generated by its identity component together with some finite group $H\leq \G_{\bar K}$. Clearly, $H$ is generated by finitely many finite constant cyclic subgroups (defined over $\bar K$). We may thus assume that $\G_{\bar K}$ is connected. 
 
Theorem 6.4.5 in \cite{Springer} implies that $\G_{\bar K}$ is generated by the centralizers of finitely many maximal tori $T\leq \G_{\bar K}$. Such a centralizer $C(T)$ is a connected closed subgroup of $\G_{\bar K}$ and it is nilpotent with maximal torus $T$ (\cite[Prop. 6.4.2]{Springer}). Let $C_u$ denote the set of unipotent elements in $C(T)$ (which is a closed, connected subgroup). Then $C=TC_u$. Now $T$ is isomorphic to a direct product of copies of $\mathbb{G}_m$ (over $\bar K$) and is thus generated by finitely many subgroups that are isomorphic to $\mathbb{G}_m$. 

It remains to show that $C_u$ is generated by finitely many subgroups that are isomorphic to $\mathbb{G}_a$. For $x \in C_u(\bar{K})$, let $G(x)\leq C_u$ denote the Zariski closure of the subgroup generated by $x$. For each $x\neq 1$, $G(x)$ is an infinite, closed, abelian, unipotent and thus also connected subgroup of $C_u$ (defined over $\bar K$).  Since $C_u$ is finite dimensional, there exist finitely many elements $x \in C_u(\bar{K})$ such that the finitely many subgroups $G(x)$ generate $C_u$. Now each of the groups $G(x)$ is isomorphic to $\mathbb{G}_a^m$ by \cite[Lemma 3.4.7.c]{Springer} for some $m \in \N$. Actually, $m=1$ since $\mathbb{G}_a^m$ does not contain a dense cyclic subgroup for $m\geq 2$. 
\end{proof}

In order to construct Picard-Vessiot rings whose differential Galois groups are given subgroups of $\GL_n$ that are isomorphic to $\mathbb{G}_a$, $\mathbb{G}_m$, or a finite cyclic group, we use the following statement that allows us to modify the representation of the differential Galois group. It is based on standard Tannakian arguments.

\begin{prop}\label{tannaka}
Let $R$ be a Picard-Vessiot ring over a differential field $F$ with field of constants~$K$, and let $G$ be its differential Galois group. Suppose that $\rho:G\rightarrow \GL_{n,K}$ is any linear representation. Then there exists a Picard-Vessiot ring $R'\subseteq R$ over $F$ and a fundamental solution matrix $Y'\in \GL_n(R')$ such that $\Gal_{Y'}^\partial(R'/F)=\rho(G)$. If moreover $\rho$ is faithful, then $R'=R$.
\end{prop}

\begin{proof} Consider the differential equation associated to $R$ and let $M$ be the corresponding differential module over $F$ (\cite{singervanderput}, discussion preceding Lemma~1.7). Let $\{\!\{M\}\!\}$ be the full subcategory of the category of differential modules over $F$ generated by $M$ (i.e., $\{\!\{M\}\!\}$ is the smallest full subcategory that contains $M$ and is closed under subquotients, finite direct sums, tensor products, and duals). Further, let $\operatorname{Repr}_G$ denote the category of finite dimensional $K$-representations of $G$. The Picard-Vessiot extension $R$ determines an equivalence of symmetric tensor categories $S:\{\!\{M\}\!\}\rightarrow \operatorname{Repr}_G$ by \cite{amano-masuoka}, Theorem~4.10.  Here, for an object $N$ in $\{\!\{M\}\!\}$, $S(N)$ is the solution space $\operatorname{ker}(\partial,R\otimes_FN)$ with $G$-action induced by the action of $G$ on $R$, and trivial on $N$. Since $R$ is a Picard-Vessiot ring for $M$, $\operatorname{dim}_K(S(N))=\operatorname{dim}_F(N)$ for all $N$; hence there exists a fundamental solution matrix $Y_N$ over $R$ generating a Picard-Vessiot ring $R_N\leq R$, and the induced morphism $G\rightarrow  \Gal_{Y_N}^\partial(R_N/F)$ is equivalent to $S(N)$. Applying this to an object 
$M'$ in $\{\!\{M\}\!\}$ such that $S(M')$ is equivalent to the given representation $\rho$ yields the first claim. If $\rho$ is faithful then the fraction fields of $R'$ and $R$ are equal by the Galois correspondence. Since a Picard-Vessiot ring is characterized as the set of differentially finite elements in a Picard-Vessiot extension (see Proposition~\ref{differentiallyfinite}), this implies $R'=R$.
\end{proof}

\subsection{Building blocks}\label{building blocks subsection}

As before, we fix a field $k$ of characteristic zero.
Next, we construct explicit extensions with differential Galois group $\Ga$, $\Gm$ and cyclic group $C_r$ of order $r$, via the following lemmas.  For the $\Ga$ case, we have:

\begin{lem} \label{Ga}
Let $c \in k((t))^\times$ and consider the derivation $\partial =c \frac{\partial}{\partial x}$ on the fields $k((x,t))\subseteq k((x))((t))$. Set $y=-\log(1- t/x):= \sum_{r=1}^\infty x^{-r}t^r/r \in k((x))((t))$. Then $R=k((x,t))[y]$ is a Picard-Vessiot ring over $k((x,t))$ with fundamental solution matrix $Y=\begin{pmatrix}
1&y\\0&1
\end{pmatrix}$, and $\Gal^\del_Y(R/k((x,t)))$ is the image of $\mathbb{G}_a$ in its two-dimensional representation $\begin{pmatrix}
1&*\\0&1
\end{pmatrix}$.
\end{lem}

\begin{proof}
Note that $y$ satisfies the differential equation
$\partial(y) 
= \frac{c}{x}\cdot\frac{1}{1-x/t} \in k((x,t))$, hence $R\subseteq k((x))((t))$ is a Picard-Vessiot ring  over $k((x,t))$ with $\Gal^\del_Y(R/k((x,t)))\leq\mathbb{G}_a$ in its two-dimensional representation (see Example \ref{example Ga}). To see that this containment is an equality, it suffices to show that the differential Galois group is nontrivial. By the Galois  correspondence, this is  equivalent to showing that $y$ does not lie in $k((x,t))$. We will show by contradiction that $y$ does not even lie in the overfield $K((x))$ of $k((x,t))$ (where $K=k((t))$ as above).

If $y = \sum_{i=m}^\infty a_i x^i$ with $a_i \in K$, then 
\[\partial(y) = \sum_{i=m}^\infty ia_i x^{i-1}.\] 
Here there is no term of degree $-1$. But the term of lowest degree in 
\[\frac{c}{x}\cdot\frac{1}{1-x/t} = cx^{-1} \bigl(1 + x/t + x^2/t^2 + \cdots\bigr)\]
has degree $-1$.  Thus $\partial(y)$ cannot equal $\frac{c}{x}\cdot\frac{1}{1-x/t}$, and this is a contradiction. 
\end{proof}

To treat the $\mathbb{G}_m$ case, we first show a generalization of Theorem~2.4 in \cite{Voelklein} (where $k$ was assumed algebraically closed).

\begin{lem} \label{structurefiniteext}
Let $\kappa$ be a field of characteristic zero, let $K = \kappa((\pi))$, and let $L$ be a finite field extension of $K$.  Then there is a finite field extension $\kappa'$ of $\kappa$ and a positive integer $e$ such that $L$ is contained in the field extension $\kappa'((\pi))(\pi^{1/e})=\kappa'((\pi^{1/e}))$ of $K$.
\end{lem}

\begin{proof}
After replacing $L$ by its Galois closure over $K$, we may assume that $L/K$ is Galois.  Let $\bar{K}$ be an algebraic closure of $K$ that contains $L$.

Let $R = \kappa[[\pi]]$, let $S$ be the integral closure of $R$ in $L$, and let $e$ be the ramification index of $S$ over the prime $\pi R$ of $R$.  The residue field of 
$R' = R[\pi^{1/e}]$ is again $\kappa$.  Let $S'$ be the integral closure of the compositum of $R'$ and $S$ inside $\bar{K}$.  Thus $R'$ and $S'$ are each complete discrete valuation rings, and $S'$ is finite over $R'$.  Let $\kappa'$ be the residue field of $S'$.

By Abhyankar's Lemma (\cite{Grothendieck}, Lemme~X.3.6), $S'$ is unramified over the prime $\pi^{1/e}R'$ of $R'$. Thus the degree of $S'$ over $R'$ is equal to the degree $d$ of the residue field extension $\kappa'/\kappa$.  Hence the inclusion $\kappa'[[\pi^{1/e}]] \hookrightarrow S'$ is an isomorphism, each ring being a degree $d$ integrally closed extension of $R'$.  It follows that $S$ is contained in $\kappa'[[\pi^{1/e}]]$ and thus that $L$ is contained in $\kappa'((\pi))(\pi^{1/e})$ as asserted.
\end{proof}

\begin{lem} \label{Gm}
Let $c \in k((t))^\times$ and consider the derivation $\partial =c \frac{\partial}{\partial x}$ on the fields $k((x,t))\subseteq k((x))((t))$.  Let 
$y=exp(t/x):= \sum_{r=0}^\infty x^{-r}t^r/r! \in k((x))((t))$. Then $R=k((x,t))[y,y^{-1}]$ is a Picard-Vessiot ring over $k((x,t))$ with $\Gal^\del_y(R/k((x,t)))=\mathbb{G}_m$.
\end{lem}

\begin{proof}
Note that $y$ satisfies the differential equation
$\partial(y) 
=  -ctx^{-2}y$ over $k((x,t))$, hence $R\subseteq k((x))((t))$ is a Picard-Vessiot ring for the one-by-one matrix $A=-ctx^{-2}$ by Proposition \ref{existencePVR} and  $\Gal^\del_y(R/k((x,t)))\leq \GL_1 = \mathbb{G}_m$. To see that this containment is an equality, it suffices by the Galois correspondence to show that $y$ is transcendental over $k((x,t))$ (since the proper closed subgroups of the multiplicative group are finite). We will show by contradiction that no finite extension of the differential field extension $k((t))((x))$ of $k((x,t))$ contains an element $y$ that satisfies
$\partial(y) 
=  -ctx^{-2}y$. (This is sufficient since derivations extend uniquely to finite extensions.)

Suppose there is a finite extension $L/k((t))((x))$ containing such a $y$. Applying Lemma~\ref{structurefiniteext} with $\kappa = k((t))$, we obtain a finite field extension $\kappa'$ of $\kappa$ and a positive integer $e$ such that $L$ is contained in $\kappa'((z))$, where $z^e=x$.  

We consider $k((t))((x))$ as a differential field with respect to $\partial =c \;\partial/\partial x$. This derivation has a unique extension to a derivation on $\kappa'((z))$ with constant field $\kappa'$, given by $\partial(z) = cz^{1-e}/e$.

Since $y \in \kappa'((z))$, we may write $y = \sum_{i=m}^\infty a_i z^i$, where each $a_i \in \kappa'$ and $a_m \ne 0$.  Now $\partial(z^i)= iz^{i-1}\partial(z) = \frac{ci}{e}z^{i-e}$, and so
\[\partial(y) = \sum_{i=m}^\infty \frac{ia_ic}{e}z^{i-e}.\] 
The coefficient of $z^j$ in this expression vanishes for $j < m-e$, and in particular for $j=m-2e$.  Meanwhile, the coefficient of $z^{m-2e}$ in $-ctx^{-2}y = -ctz^{-2e}y$ is $-cta_m$, which is non-zero.  Thus $\partial(y)$ cannot equal $-ctx^{-2}y$.  This is a contradiction.
\end{proof}

\begin{lem}\label{cyclic}
Let $r$ be a positive integer, and assume that $k$ contains a primitive $r$-th root of unity. Then $y := x\bigl(1 - x^{-1}t\bigr)^{1/r}$ is contained in $k((x))((t))$ and $R=k((x,t))[y]\subseteq k((x))((t))$ is a Picard-Vessiot ring over $k((x,t))$  with $\Gal^\del_y(R/k((x,t)))=C_r$ in its one-dimensional representation. 
\end{lem}

\begin{proof}
First note that the element $y/x = \bigl(1 - x^{-1}t\bigr)^{1/r}$ lies in the subring 
$k[x^{-1}][[t]]$ of $k((x))((t))$ because it is a power series in $x^{-1}t$, hence $y \in k((x))((t))$.

Now $y^r$ lies in $k[[x,t]]$.  We claim that $r$ is minimal for this property.  To show this, let $m$ be minimal such that $y^m \in k[[x,t]]$.  Then $m$ divides $r$; let $d = r/m$.  Thus $y^r = x^r-x^{r-1}t$
is a $d$-th power of some element $f$ in $k[[x,t]]$.  
Since $f^d = x^r-x^{r-1}t$, it follows
that $f$ has no term of total degree less than $m$; and among its terms of total degree $m$ there is a non-vanishing $x^m$ term as well as some $x^i t^j$ term with $i+j=m$ and $j>0$.  
Let $j$ be maximal for this property.  Then $x^r-x^{r-1}t = f^d$ has a non-vanishing $x^{id}t^{jd}$ term.  Hence $i=r-1$ and $j=d=1$.  Thus $r=m$, as claimed.  

Since $k$ contains a primitive $r$-th root of unity, it then follows from Kummer theory that $k((x,t))(y)$ is a cyclic extension of $k((x,t))$ of degree $r$. The derivation $\del$ extends uniquely to $R=k((x,t))(y)=k((x,t))[y]$, and $y$ satisfies the differential equation $\del(y)=ay$ for $a:=\del(y)/y=\del(y^r)/(ry^r) \in k((x,t))$. Hence $R$ is a Picard-Vessiot ring over $k((x,t))$ by Proposition~\ref{existencePVR}, and $\Gal^\del_y(R/k((x,t)))=C_r$. 
\end{proof}

\begin{prop} \label{bldg block extens} Let $k$ be a field of characteristic zero. Let $P$ be a $k$-point of the affine $x$-line $\mathbb A^1_{k} \subset \mathbb P^1_{k}$, and let $F_P$ and $F_{\wp(P)}$ be as defined in Section \ref{sectionpatching}. Endow these fields with the derivation $c\frac{\del}{\del x}$ for some $c \in k((t))^\times$. Suppose that $\G \subseteq \GL_n$ is a linear algebraic group over $K=k((t))$ that is $K$-isomorphic to $\mathbb G_m$, $\mathbb G_a$, or $C_r$ where $r$ is such that $k$ contains a primitive $r$-th root of unity.  
Then there exist a matrix $A \in F_P^{n \times n}$, a Picard-Vessiot ring $R/F_P$ for $A$
that is contained in $F_{\wp(P)}$, and a fundamental solution matrix 
$Y \in \GL_n(R)$ for $A$, such that $\Gal^\del_Y(R/F_P)=\G$.
\end{prop}

\begin{proof}
After a change of variables, we may assume that $P$ is the point of $\mathbb P^1_{k}$ where $x=0$. We then have inclusions $F_P=k((x,t)) \subset F_{\wp(P)}=k((x))((t))$.
By Lemma~\ref{Ga},~\ref{Gm}, and~\ref{cyclic}, there exists a Picard-Vessiot ring over $F_P$ contained in $F_{\wp(P)}$ with differential Galois group ${\mathbb G}_m$, ${\mathbb G}_a$, or $C_r$ as in the statement, respectively. The result then follows from Proposition~\ref{tannaka}.
\end{proof}

\end{section}

\begin{section}{The differential inverse Galois problem} \label{section inverse}

In this section, we solve the inverse problem over function fields over $k((t))$, where as before $k$ is any field of characteristic zero. We begin by making the following 

\begin{Def}\label{occurs}
Let $F$ be a differential field with field of constants $K$. We say that a
linear algebraic group $G$ defined over $K$ {\bf is a differential
Galois group} over $F$ if there exists a Picard-Vessiot ring $R/F$ with
Galois group $K$-isomorphic to $G$. 
\end{Def}
Note that if $G$ is a differential Galois group over $F$, this also implies that every faithful representation $\G\leq \GL_n$ of $G$ arises as a differential Galois group over $F$ by Proposition~\ref{tannaka}; namely, there exists a $Y\in\GL_n(R)$ such that $\Gal^\del_Y(R/F)=\G$.

We next note that in order to solve the inverse problem, we may modify the derivation.

\begin{lem}\label{change of derivation}
Let $F$ be a field, and let $\del$ and $\del'$ be derivations on $F$ such that there is an $a \in F^\times$ with $\del'=a\del$. Let $\G\leq \operatorname{GL}_n$ be a linear algebraic group over $C_F$, and suppose that $\G$ is a differential Galois group over $(F,\del)$. Then $\G$ is also a differential Galois group over $(F,\del')$.
\end{lem}

\begin{proof}
By assumption, there exists a matrix $A \in F^{n\times n}$ with Picard-Vessiot ring $R$ over $(F,\del)$ and fundamental solution matrix $Y \in \GL_n(R)$ such that $\Gal^\del_Y(R/F)=\G$ in the given representation $\G\leq \GL_n$. Then $\del'=a\del$ extends to a derivation on $R$ with $\del'(Y)=aAY$. By Proposition~\ref{existencePVR}, $R$ is a Picard-Vessiot ring over $(F,\del')$ for the matrix $aA$ with fundamental solution matrix $Y$. Note that $\underline{\Aut}^{\del'}(R/F)=\underline{\Aut}^{\del}(R/F)$ and thus $\Gal^{\del'}_Y(R/F)=\Gal^\del_Y(R/F)=\G$.
\end{proof}

\begin{subsection}{The inverse problem over $k((t))(x)$}

As before, $k$ denotes a field of characteristic zero. In this section, we show that every linear algebraic group over $k((t))$ is a differential Galois group over the rational function field $k((t))(x)$, equipped with any nontrivial derivation with field of constants $k((t))$.

\begin{lem}\label{distinctconjugates}
Let $k_0\leq k$ such that $k/k_0$ is a finite Galois extension of degree $d$. Assume further that $k$ contains a primitive $e$-th root of unity $\zeta$ for some $e \in  \N$. Then there exist infinitely many elements $a \in k$ such that there are $d\cdot e$ distinct elements among the $\Galf(k/k_0)$-conjugates of $a,\zeta a, \dots, \zeta^{e-1}a$. 
\end{lem}

\begin{proof}
We may assume that $d > 1$.  Since $k/k_0$ is separable, there exists some primitive element $b\in k$. Thus there are exactly $d$ conjugates of $b$. For any $c \in k_0$, consider the $\Galf(k/k_0)$-conjugates of $\zeta^i(c+b)$ for $0 \leq i < e$. If the number of conjugates is less than $d\cdot e$, then there exists a non-trivial $\sigma \in \Galf(k/k_0)$ such that $\zeta^i(c+b)=\sigma(\zeta^j(c+b))$ for some $0\leq i\leq j<e$. Note that $\zeta^i\neq \sigma(\zeta^j)$, since $\sigma(c)=c$ and $\sigma(b)\neq b$. Hence there exists $0<l<e$ with $\frac{\zeta^i}{ \sigma(\zeta^j)}=\zeta^l$. We obtain $c(\zeta^l-1)=\sigma(b)-\zeta^l\cdot b$, hence $c$ is contained in the finite set $S=\{\frac{\sigma(b)-\zeta^l\cdot b}{\zeta^l-1}\mid 0<l<e, 1\neq\sigma\in \Galf(k/k_0) \}\subseteq k$. Therefore, almost all choices of $c$ yield an element $a=c+b$ with the desired property. 
\end{proof}

Suppose that $k/k_0$ is a finite Galois extension of degree $d$ and that $k$ contains a primitive $e$-th root of unity, for a positive integer $e$. Consider the Laurent series field extension $k((t))/k_0((t_0))$ with $t_0:=t^e$. Define $\Gamma=\Galf(k((t))/k_0((t_0)))$. Recall 
from Example~\ref{example action}(a) that $\Gamma$ has order $d\cdot e$, and it is a semi-direct product of the cyclic group $\Galf\bigl(k((t))/k((t_0))\bigr)$ of order $e$ with the group $\Galf(k/k_0)$. The action of $\Gamma$ on $K=k((t))$ over $K_0=k_0((t_0))$ extends to an action of $F=K(x)$ over $K_0(x)$ by taking $x$ to $x$. Consider the variable change $z=x/t$. 

\begin{lem} \label{properaction}
In the above setup, consider the induced action of $\Gamma$ on the $z$-line $\mathbb{P}^1_k$ (as explained in Example \ref{example action}(b)). Then for any positive integer $r$, there exist $r$ $k$-points $P_1,\dots, P_r \in \mathbb A^1_{k} \subset \mathbb P^1_{k}$ whose orbits $P_1^\Gamma, \dots P_r^\Gamma$ are disjoint and are each of order $|\Gamma|$.
\end{lem}

\begin{proof}
 For each $\sigma \in \Gamma$, there exists an integer $0\leq n_\sigma \leq e-1$ with $\sigma(t)=\zeta^{n_\sigma}t$ and thus $\sigma(z)=\zeta^{-n_\sigma}z$.  Recall from Example \ref{example action}(b) that if $\sigma \in \Gamma$ and $P \in \Pcal$ is the $k$-point $z=b$, then $P^{\sigma^{-1}}$ is the point $z=\zeta^{n_\sigma}\sigma(b)$.
 
  We apply Lemma \ref{distinctconjugates} and obtain an $a_1 \in k$ such that there are $d\cdot e$ distinct elements among the $\Galf(k/k_0)$-conjugates of $\zeta^ia_1$ ($0 \leq i \leq e-1$). Let $P_1$ be the point $z=a_1$. Then $P_1^\Gamma$ consists of the points  $z=\zeta^{n_\sigma}\sigma(a_1)$. We claim that these $|\Gamma|=de$ points are distinct.  It suffices to check that 
\[\{\zeta^{n_\sigma}\sigma(a_1) \ | \ \sigma \in \Gamma \}=\{\tau(\zeta^ia_1) \ | \  0 \leq i < e, \tau \in \Galf(k/k_0) \}.\] 
Given an element $\tau(\zeta^ia_1)$ in the right hand side,
let $\psi$ be the element of $\Galf\bigl(k((t))/k_0((t))\bigr) \le \Gamma$ that lifts $\tau \in \Galf(k/k_0)$.
Thus $n_\psi=0$.  
Since $\tau(\zeta^i)$ is an $e$-th root of unity, there is an element $\pi$ in the cyclic group $\Galf\bigl(k((t))/k((t_0))\bigr)$ with $\zeta^{n_\pi}=\tau(\zeta^i)$. 
Define $\sigma=\pi\circ \psi \in \Gamma$. As $n_\psi=0$, $\zeta^{n_\sigma}=\zeta^{n_\pi}
=\tau(\zeta^i)$. Therefore $\zeta^{n_\sigma}\sigma(a_1)=\tau(\zeta^ia_1)$, and the claim follows.
  
To find a second point $P_2$, we apply Lemma \ref{distinctconjugates} to obtain an element $a_2 \in k$ with the same property and such that $a_2$ is not contained in the finite set of $\Galf(k/k_0)$-conjugates of $\zeta^ia_1$. Therefore, $P_2 \notin P_1^\Gamma$. Thus $P_1^\Gamma$ and $P_2^\Gamma$ are disjoint orbits of order $|\Gamma|$. By induction, we get points $P_1,\dots,P_r$ as asserted.
\end{proof}

\begin{thm}\label{result}
Let $k$ be a field of characteristic zero, let $K=k((t))$, and let $F$ be a rational function field of transcendence degree one over $K$. Let $\del$ be a non-trivial derivation on $F$ with field of constants $K$. Then every linear algebraic group defined over $K$  is a differential Galois group over $F$. 
\end{thm}

\begin{proof}
In order to make the notation in the proof consistent with the notation used in Section~\ref{sectionpatching}, we rename our base field as follows: We consider the rational function field $F_0=K_0(x)$ over $K_0=k_0((t_0))$, the field of formal Laurent series over a field $k_0$ of characteristic zero. By Lemma~\ref{change of derivation} we may assume without loss of generality that the derivation $\del$ on $F_0$ satisfies $\del(x)=1$, with $C_{F_0}=K_0$ (hence $\del=\del/\del x$).

Let $\G\leq \GL_n$ be a linear algebraic group defined over $K_0$. We will show that there exists a differential equation $A \in F_0^{n\times n}$ with a Picard-Vessiot ring $R$ and fundamental solution matrix $Y$ such that $\Gal^\del_Y(R/F_0)=\G$.

 By Proposition~\ref{groupsplits}, there exists a finite extension $K/K_0$ and finitely many $K$-subgroups $\G_1,\dots,\G_r \le \G_K$ that generate $\G_K$ and such that each $\G_i$ is either a finite cyclic group or is $K$-isomorphic to either $\mathbb{G}_m$ or $\mathbb{G}_a$.  After enlarging $K$, we may assume it contains a primitive $m$-th root of unity for each $m$ that is the order of one of the finite cyclic groups among $\G_1,\dots,\G_r$. We may further assume that $K=k((t))$ for a finite Galois extension $k/k_0$ and an $e$-th root $t$ of $t_0$ for some $e \in \N$, by Lemma~\ref{structurefiniteext}. Without loss of generality, we may assume that $k$ contains a primitive $e$-th root of unity $\zeta$. Then $K/K_0$ is a finite Galois extension (see Example \ref{example action}) and we set $\Gamma=\operatorname{Gal}(K/K_0)$.

 We first construct a Picard-Vessiot ring over $F=K(x)$ with Galois group~$\G_K$, using patching. 
 Let $z=x/t$. By Lemma~\ref{properaction}, we can fix $k$-points $P_1,\dots,P_r$ on the affine $z$-line $\mathbb{A}^1_k \subset \mathbb{P}^1_k$ such that $|P_i^\Gamma|=|\Gamma|$ for all $i$ and such that these orbits are disjoint. We set $\Pcal=P_1^\Gamma \cup \dots \cup P_r^\Gamma$. Let $(\Pcal, \Bcal, U)$ be the corresponding differential patching data as explained in the beginning of Section \ref{sectionpatching}. As explained in Example~\ref{example action}(b), $\Gamma$ acts on the differential patching data $(\Pcal, \Bcal, U)$ (see also Definition \ref{actonpatchingdata}).
   
 By Proposition~\ref{bldg block extens} (with $c=\del(z)=t^{-1}$), we can now fix Picard-Vessiot rings $R_{P_j}/F_{P_j}$ for each $j=1,\dots,r$ with fundamental solution matrices $Y_{P_j} \in \GL_n(F_{\wp({P_j})})$ such that $\Gal^\del_{Y_{P_j}}(R_{P_j}/F_{P_j})=\G_j \le \G_K \le \GL_{n,K}$. Now let $P \in \Pcal$ be arbitrary. Then there exists a unique $\sigma \in \Gamma$ and a unique integer $1\leq j \leq r$ such that $P^\sigma=P_j$. Recall that there is a differential isomorphism $\sigma \colon F_{\wp(P_j)} \to F_{\wp(P)}$ restricting to $\sigma \colon F_{P_j} \to F_P$ (see Definition~\ref{actonpatchingdata} and Example \ref{example action}). Set $Y_P=\sigma(Y_{P_j}) \in \GL_n(F_{\wp(P)})$ and $R_P=F_P[Y_P,Y_P^{-1}]=\sigma(R_{P_j})\subseteq F_{\wp(P)}$. We define $A_P=\del(Y_P)Y_P^{-1}$. Note that $A_P=\sigma(A_{P_j}) \in \GL_n(F_P)$. Hence $R_P$ is a Picard-Vessiot ring over $F_P$ for $A_P$ with fundamental solution matrix $Y_P$. Fix an extension of $\sigma$ from $K$ to $\bar{K}$. Then $\sigma$ induces a differential isomorphism $\sigma \colon R_{P_j}\otimes_K \bar{K} \to R_P\otimes_K \bar{K}$. Hence \[\underline{\Aut}^\del(R_P/F_P)(\bar{K})= \sigma\underline{\Aut}^\del(R_{P_j}/F_{P_j})(\bar{K})\sigma^{-1}\] and thus 
 \begin{eqnarray*}
   \Gal^\del_{Y_P}(R_P/F_P)(\bar{K})&=&\Gal^\del_{\sigma(Y_{P_j})}(\sigma(R_{P_j})/\sigma(F_{P_j}))(\bar{K})\\
   &=&\sigma(\Gal^\del_{Y_{P_j}}(R_{P_j}/F_{P_j})(\bar{K}))\\
   &=&\sigma(\G_j(\bar{K})).
 \end{eqnarray*}
 We conclude by Remark~\ref{remark equal groups} that $\Gal^\del_{Y_P}(R_P/F_P)=\G_j^\sigma$, the linear algebraic group obtained from $\G_j$ by applying $\sigma$ to the defining equations. \\
 
Theorem \ref{diffpatching} then yields a matrix $A \in F^{n\times n}$ and a Picard-Vessiot ring $R=F[Y,Y^{-1}]$ for $A$ over $F$ ($Y \in \GL_n(F_U)$ a fundamental solution matrix) such that \[\Gal^\del_Y(R/F)=\bar{<\G_j^\sigma \ | \  1\leq j \leq r, \sigma \in \Gamma > }=\bar{<\G_K^\sigma \ | \ \sigma \in \Gamma>}=\G_K.\] 
 Moreover, we constructed the fundamental solution matrices $(Y_P)_{P \in \Pcal}$ such that $\sigma(Y_{P^\sigma})=Y_P$ for all $P \in \Pcal$, $\sigma \in \Gamma$. Therefore, Theorem~\ref{diffpatching}\ref{diffpatching2} asserts that we can assume that $Y$ is $\Gamma$-stable, which in turn implies that all entries of $A$ are contained in $F^\Gamma=F_0$. Therefore, $R_0 := F_0[Y,Y^{-1}]$ is a Picard-Vessiot ring for $A$ over $F_0$ with differential Galois group $\Gal^\del_Y(R_0/F_0)
= \G \le \GL_{n,K_0}$ (via Lemma~\ref{lemmainvariantPVR} with $L=F_U$). 
\end{proof}

\begin{prop}
The Picard-Vessiot ring $R$ in Theorem \ref{result} can be constructed such that $R \subseteq K((x))$.
\end{prop}

\begin{proof}
We switch to the notation as in the proof of Theorem \ref{result}, so we need to show that $R_0\subseteq K_0((x))=k_0((t_0))((x))$. We refine the choice of $\Pcal$ such that the point $z=0$ is not contained in $\Pcal$. Set $m=|\Gamma|\cdot r$, where $r$ is as in the proof of Theorem \ref{result}, and let $\alpha_1,\dots,\alpha_m \in k^{\times}$ such that $\Pcal$ is the set of all points $z=\alpha_i$. Now $R_0=F_0[Y,Y^{-1}]$ and $Y \in \GL_n(F_U)$ is $\Gamma$-invariant. As all $\alpha_i$ are non-zero, 
\[ F_U=\Frac([(z-\alpha_1)^{-1},\dots,(z-\alpha_m)^{-1}][[t]])\subseteq \Frac(k[[z,t]]).\]
Moreover $k[[z,t]]$ is contained in $k((t))[[x]]$, as can be seen by replacing $z$ by $x/t$ in any power series in $k[[z,t]]$.  Thus $\Frac(k[[z,t]])\subseteq \Frac(k((t))[[x]]) = k((t))((x))$. Hence $F_U$ can be regarded as a differential subfield of $k((t))((x))$;  and it is easy to check that the embedding $F_U\subseteq k((t))((x))$ is $\Gamma$-equivariant. Therefore, all entries of $Y$ lie in $F_U^\Gamma\subseteq k((t))((x))^\Gamma=k_0((t_0))((x))$ and the assertion follows.
\end{proof}

\begin{rem}
Theorem~\ref{result}, in conjunction with \cite{Hrush}, can be used to obtain a new proof of the fact that every linear algebraic group $G$ over $\bar{\Q}$ is a differential Galois group over $\bar\Q(x)$.  (In \cite{HartCrelle}, the second author gave an earlier proof of this, which also handled the case of $C(x)$ for $C$ an arbitrary algebraically closed field of characteristic zero). Namely, we may regard $G$ as defined over $\bar{\Q}((t))$; and then by Theorem~\ref{result}, there is a Picard-Vessiot extension $R$ of $\bar{\Q}((t))(x)$ with differential Galois group $G_{\bar{\Q}((t))}$, for some matrix $A \in \bar{\Q}((t))(x)^{n \times n}$.  
This data descends to some finitely generated $\bar{\Q}$-subalgebra $D$ of $\bar{\Q}((t))$; and by 
\cite[Section~V.1]{Hrush}, infinitely many specializations to closed points of $\Spec(D)$ yield the same differential Galois group $G$.  Since $\bar \Q$ is algebraically closed, these closed points are all $\bar \Q$-rational; and so $G$ is a differential Galois group over $\bar \Q(x)$.
\end{rem}

\begin{rem}
Theorem~\ref{result} in particular asserts that every finite \'etale group scheme $G$ over a characteristic zero Laurent series field $K=k((t))$ is a differential Galois group over $K(x)$.  Equivalently, for each such $G$, there is a finite morphism of smooth connected projective $K$-curves $C \to \mathbb{P}^1_K$, with a faithful action of $G$ on $C$, such that the generic point of $C$ is a $G$-torsor over $K(x)$.  In the case that $G$ is a finite {\em constant} group, this says that $G$ is a Galois group over $K(x)$; and that was proven (in fact in arbitrary characteristic) in 
\cite[Theorem~2.3]{HarGCAL}, using formal patching methods.  In the case of finite \'etale group schemes $G$ that need not be constant, this was shown (again in arbitrary characteristic) in \cite{MB01}.  That paper in fact showed more, viz.\ that this holds if $K$ is more generally a large field in the sense of Pop (\cite{Pop}); and also that one fiber of the morphism $C \to \mathbb{P}^1_K$ can be given in advance; this extended results of Pop (\cite[Main Theorem~A]{Pop}) and Colliot-Th\'el\`ene (\cite[Theorem~1]{CT00}) for finite constant groups.
\end{rem}

\end{subsection}

\begin{subsection}{Passage to finitely generated extensions}
In this subsection, we prove our main result (Theorem~\ref{mainresult}). This is a consequence of Theorem~\ref{result}, together with the following result (see Theorem~\ref{generalresult}): If $F$ is any differential field of characteristic zero with the property that every linear algebraic group defined over $C_F$ is a differential Galois group over $F$, then every finitely generated extension $L/F$ of differential fields with $C_L=C_F$ also has this property. To prove the latter, we first require some preparation.

\begin{lem}\label{constants indep} 
Let $L/F$ be an extension of differential fields. If $x_1,\dots,x_n \in C_L$ are algebraically independent over $C_F$, then they are algebraically independent over $F$.
\end{lem}

\begin{proof}
We prove the contrapositive.
Assume that $x_1,\dots,x_n\in C_L$ are algebraically dependent over $F$. Let $r$ be the smallest positive integer such that there exists a polynomial expression over $F$, vanishing on $(x_1,\dots,x_n)$, with exactly $r$ monomial terms. Let $p$ be such a polynomial. After dividing by an element in $F^{\times}$, we can assume that some coefficient equals one. Differentiating the coefficients of $p$ yields a polynomial with less than $r$ monomial terms that also vanishes on $(x_1,\dots,x_n)$ and must be the zero polynomial by minimality of~$p$. Hence all coefficients of $p$ are contained in $C_F$, and thus $x_1,\dots,x_n$ are algebraically dependent over $C_F$.  
\end{proof}

\begin{lem}\label{compositum}
Let $F$ be a differential field, let $R/F$ be a Picard-Vessiot ring and set $G=\underline{\Aut}^\del(R/F)$ and $E=\Frac(R)$. Let $H_1,\dots,H_r$ be closed subgroups of $G$ (defined over $C_F$) and set $H=\bigcap\limits_{i=1}^r H_i$. Then 
$$E^{H_1}\cdots E^{H_r}=E^H\subseteq E.$$  
\end{lem}

\begin{proof}
The compositum of $E^{H_1}, \dots, E^{H_r}$ is a differential subfield of $E$, so it equals $E^{\tilde H}$ for some closed subgroup $\tilde H$ of $G$, by the Galois correspondence. Then $E^{H_i}\subseteq E^{\tilde H}$, so  $\tilde H\subseteq H_i$ for all $i$, and thus $\tilde H \subseteq H$. On the other hand, every element in the compositum is invariant under $H$, hence $\tilde H \supseteq H$.
\end{proof}

\begin{lem}\label{intersection tensor}
Let $F_1,\dots,F_r$ be finite field extensions of a field $F$. Embed $F$ into $F_1\otimes_F \cdots \otimes_F F_r$ via $1\mapsto 1\otimes \cdots\otimes 1$.  Let $V \subseteq F_1\otimes_F \cdots \otimes_F F_r$ be an $F$-subspace of dimension $d<r$ with $F \subseteq V$.  If $1 \le s \le r-d+1$, then after relabeling the indices, \[(F_1\otimes_F  \cdots \otimes_F F_s)\cap V=F,\]  where the intersection is taken inside $F_1\otimes_F \cdots \otimes_F F_r$.
\end{lem}

\begin{proof}
Since $F$ is contained in each $F_i$, the right hand side is contained in the left hand side.  For the other containment, it suffices to prove the assertion for $s=r-d+1$.

For each $i \leq r$, fix an $F$-basis $\{v_{i,j}\mid 1 \leq j \leq [F_i:F] \}$ of $F_i$ with $v_{i,1}=1$. We use induction on $r$. If $r=2$, then $d=1$ and $V=F$, and the claim follows. Let $r>2$. If $(F_1\otimes_F \cdots \otimes_F F_{r})\cap V=F$, there is nothing to prove. Otherwise, fix an $x \in (F_1\otimes_F \cdots \otimes_F F_{r})\cap V$ that is not contained in $F$. Then we can write $x$ as an $F$-linear combination of the elements $v_{1,j_1}\otimes\dots \otimes v_{r,j_{r}}$ in a unique way. Since $x \notin F$, this sum contains a term $v_{1,j_1}\otimes\dots \otimes v_{r,j_{r}}$ with $j_i\neq 1$ for some $i$. After relabeling the indices we may assume $i=r$, which implies $x \notin F_1\otimes_F \cdots \otimes_F F_{r-1}$. Hence  $\dim_F((F_1\otimes_F \cdots \otimes_F F_{r-1})\cap V)\leq d-1$ and the claim follows by induction.
\end{proof}

As before, let $F$ be a differential field, and let $G$ be a linear algebraic group defined over $K=C_F$. We now prove the above mentioned general result, using a strategy sometimes called the {\em Kovacic trick}: Assume there is a Picard-Vessiot extension $R$ over $F$ with differential Galois group $G^r$. The idea is to show that $R\otimes_{F} L$ must contain a Picard-Vessiot ring over $L$ with differential Galois group $G$ if $r$ is sufficiently large.

\begin{thm}\label{generalresult}
Let $F$ be a differential field of characteristic zero and write $K=C_F$. Let $L/F$ be a differential field extension that is finitely generated over $F$ with $C_L=K$. Let $G$ be a linear algebraic group defined over $K$ with the property that for every $r \in \N$, $G^r$ is a differential Galois group over $F$. Then $G$ is a differential Galois group over $L$. 
\end{thm}

\begin{proof}
Let $L'$ be the algebraic closure of $F$ in $L$ and let $L''$ be the normal closure of $L'$ in $\bar{F}$. Set $d=[L'':F]$ and $m=\trd(L/F)+1$. Note that $1\leq d,m <\infty$. \medskip 

\noindent\textit{First step:} We show that there is a Picard-Vessiot ring $R/F$ with differential Galois group $\Autd(R/F)$ isomorphic to $G^{2m}$, and with the following property: If $F'$ denotes the algebraic closure of $F$ in $\Frac(R)$, then (1) $F'\otimes_F L'$ is a field in which (2) $K$ is algebraically closed.

By assumption, there is a Picard-Vessiot ring $R/F$ with differential Galois group $G^{2m+2d}$. We show that we can achieve (1) and (2) by replacing $R$ with a suitable subring.

Set $E=\Frac(R)$ and let $E_i$ be the fixed field $E^{G\times \dots \times G \times 1 \times G\times \dots \times G}$ (omitting the $i$-th factor) for $1 \leq i \leq 2m+2d$. 
Thus $C_E=C_{E_i}=K$.
 Since $G \times \dots \times G \times 1 \times G\times \dots \times G$ is a normal subgroup of $G^{2m+2d}$, the Galois correspondence implies that $E_i=\Frac(R_i)$ for a Picard-Vessiot ring $R_i/F$ with differential Galois group $\Autd(R_i/F)\cong G$ for each $i$. 

We let $F'\subseteq E$ be the algebraic closure of $F$ in $E$ and similarly let $F_i'\subseteq E_i$ the algebraic closure of $F$ in $E_i$. Then Lemma \ref{closed} implies $F'=E^{G^0\times \dots \times G^0}$ and $F_i'=E^{G\times \dots \times G \times G^0 \times G\times \dots \times G}$. By Lemma~\ref{compositum}, the compositum $F_1'\cdots F_{2m+2d}'$ equals $E^{G^0\times \dots \times G^0}=F'$. Again by Lemma~\ref{closed}, $[F':F]=|G(\bar{K})/G^0(\bar{K})|^{2m+2d}$ and  $[F_i':F]=|G(\bar{K})/G^0(\bar{K})|$ for each $1 \leq i \leq 2m+2d$. A dimension count yields 
\[F'\cong F_1'\otimes_F\cdots \otimes_F F_{2m+2d}'.\] 
Let $V=F'\cap L''$.  Thus $\dim(V) \le \dim(L'') = d$.
An application of Lemma~\ref{intersection tensor} yields $(F_1'\otimes_F\cdots \otimes_F F_{2m+d}')\cap L''=F$, possibly after renumbering the indices. 

We now change notation: we replace $E$ by $E^{1\times \cdots \times 1\times G \times \cdots \times G}$ (with $d$ copies of $G$ on the right); we replace $R$ with the unique Picard-Vessiot ring in $E^{1\times \cdots \times1 \times G \times \cdots \times G}$ (hence $\Autd(R/F)\cong G^{2m+d}$); and we replace $F'$ by the algebraic closure $F_1'\otimes_F\cdots \otimes_F F_{2m+d}'$ of $F$ in $E^{1\times \cdots 1\times G \times \cdots \times G}$. (Note that this tensor product is indeed equal to the relative algebraic closure, since both have the same degree over $F$ by Lemma~\ref{closed} and since the former field extension is contained in the latter.)
We similarly replace $E_i$ and $R_i$.  In this new situation,
$F'\cap L''=F$. As $L''/F$ is a finite Galois extension, $F' \cap L''=F$ implies that $F'\otimes_F L''$ is a field. In particular, $F'\otimes_F L'$ is a field which proves (1) for this choice of $R$.

To establish (2), let $K'$ be the algebraic closure of $K$ in the field $F'\otimes_F L'$. Thus $K'=C_{F'\otimes_F L'}$.  Moreover, $K'$ and $F'$ are
linearly disjoint over $K$, since $K'/K$ is algebraic whereas 
$K$ is algebraically closed in $F'$ (using $C_{F'}\subseteq C_E=K$).
Hence a $K$-basis of $K'$ is linearly independent over $F'$.  
Therefore, $[K':K]=[F'K':F']\leq [F'\otimes_FL':F']=[L':F]\leq d$ and thus $[K'L':L']\leq d$, where all composita are taken inside $F'\otimes_F L'$. Recall that $F'\cong F_1'\otimes_F\cdots \otimes_F F_{2m+d}'$, hence  \[F'\otimes_F L'\cong (F_1'\otimes_F L')\otimes_{L'}\dots\otimes_{L'} (F_{2m+d}'\otimes_F L'). \] 
Since the $L'$-vector space $V=L'K'$ has dimension at most $d$, 
by Lemma \ref{intersection tensor} we can relabel the indices such that \[(F_1'\otimes_F L')\otimes_{L'}\dots\otimes_{L'} (F_{2m}'\otimes_F L')\cap L'K'=L'.\] In particular $(F_1'\otimes_F L')\otimes_{L'}\dots\otimes_{L'} (F_{2m}'\otimes_F L')\cap K'=L'\cap K'=K,$ where the last equality uses that $C_L=K$. 

We again change notation, 
replacing $E$ by $E^{1\times \cdots \times 1 \times G \times \cdots \times G}$ (with $d$ copies of $G$ on the right) and replacing $R$, $F'$, $E_i$, and $R_i$ correspondingly as above.  In this new notation, we obtain that $F'\otimes_F L'$ is a field in which $K$ is algebraically closed. Note that $\Autd(R/F)\cong G^{2m}$ for this new $R$.  \medskip 

\noindent\textit{Second step:} We claim that if $R$ is as constructed in the first step and $E=\Frac(R)$, then $E\otimes_F L$ is an integral domain and $K$ is algebraically closed in $\tilde E=\Frac(E\otimes_F L)$. 

As $E/F'$ and $L/L'$ are regular, both $E\otimes_{F'}(F'\otimes_{F}L')$ and $(F'\otimes_{F}L')\otimes_{L'}L$ are regular over $F'\otimes_{F}L'$. Therefore, their tensor product over $F'\otimes_{F}L'$ is regular over $F'\otimes_{F}L'$ (\cite[Proposition V.17.3b]{Bourbaki}). But this tensor product is isomorphic to $E\otimes_F L$, and we conclude that $E\otimes_F L$ is a regular $(F'\otimes_{F}L')$-algebra. In particular, $E\otimes_F L$ is an integral domain and its fraction field $\tilde E$ is a regular extension of the field $F'\otimes_{F}L'$ (\cite[Proposition V.17.4]{Bourbaki}).
   
Recall that $K$ is algebraically closed in the field $F'\otimes_{F}L'$ (by the first step). 
But $F'\otimes_{F}L'$ is algebraically closed in $\tilde E$, as $\tilde E$ is regular over $F'\otimes_{F}L'$.  We conclude that $K$ is algebraically closed in ${\tilde E}$ which completes the second step. \medskip 

\noindent\textit{Third step:}  With notation as above,  for $1\leq i \leq 2m$
we define $\tilde R_i=R_i \otimes_F L$; this is an integral domain, being contained in $\tilde E$.  We also let $\tilde E_i=\Frac(\tilde R_i)$.  Thus $\tilde E_i$ is the compositum of $E_i$ and $L$ inside $\tilde E$. We claim that $C_{\tilde E_i}=K$ for some $i$.

Suppose to the contrary that $C_{\tilde E_i}$ strictly contains $K$ for all $1\leq i \leq 2m$. Then for each $i$, $\tilde E_i$ contains a constant that is transcendental over $K$ (since $K$ is algebraically closed in $\tilde E_i\subseteq \tilde E$ by the second step) and thus transcendental over $L$ (by Lemma \ref{constants indep}, using $C_L=K$). Therefore, $\trd(C_{\tilde E_i}L/L)\geq 1$ for all $i$ and thus \[\trd(\tilde E_i/C_{\tilde E_i}L)\le \trd(\tilde E_i/L)-1\le \trd(E_i/F)-1=\dim(G)-1,\] where the last equality follows from the fact that $R_i/F$ is a Picard-Vessiot ring with differential Galois group $G$. Base change from $C_{\tilde E_i}L$ to $C_{\tilde E}L$ then yields $\trd(\tilde E_iC_{\tilde E}/C_{\tilde E}L)\le \dim(G)-1$ for all $1 \leq i \leq 2m$. By Lemma \ref{compositum}, the compositum $E_1\cdots E_{2m}$ equals $E$. Therefore, $\tilde E$ equals the compositum of $\tilde E_1,\dots, \tilde E_{2m}$ inside $\tilde E$ and we conclude that 
\begin{equation}\label{eqtr1}
\trd(\tilde E/C_{\tilde E}L)\leq 2m(\dim(G)-1).
\end{equation}
 
On the other hand, any subset of $C_{\tilde E}$ that is algebraically independent over $K$ remains algebraically independent over $E$ by Lemma \ref{constants indep}, since $C_E=K$. Therefore, $\trd(C_{\tilde E}/K)=\trd(EC_{\tilde E}/E) \le \trd(\tilde E/E)$ and we obtain
 \begin{equation}\label{eqtr2}
\trd(C_{\tilde E}L/L)\le \trd(C_{\tilde E}/K) \le \trd(\tilde E/E)\le \trd(L/F)=m-1,
 \end{equation} where we used that $\tilde E$ is a compositum of $E$ and $L$ to obtain the last inequality.
Equations (\ref{eqtr1}) and (\ref{eqtr2}) together yield $\trd(\tilde E/L)\leq 2m\dim(G)-m-1$, and therefore 
\[\trd(E/F)\le\trd(\tilde E/F)\leq 2m\dim(G)-2. \] But $R/F$ is a Picard-Vessiot ring with differential Galois group $G^{2m}$, so $\trd(E/F)=2m\dim(G)$, a contradiction, which completes the third step. \smallskip 

{\em Conclusion of the proof:}
We conclude that for some $1\le i \le 2m$, $\tilde R_i=R_i \otimes_F L$ is an integral domain and $\tilde E_i=\Frac(\tilde R_i)$ satisfies $C_{\tilde E_i}=K$. By Proposition \ref{extending PVR}, $\tilde R_i$ is a Picard-Vessiot ring over $L$ with $\Autd(\tilde R_i/L)\cong\Autd(R_i/F)\cong G$. 
\end{proof}

A field of characteristic zero with a nontrivial derivation is always a transcendental extension of its field of constants (since derivations extend uniquely to separable algebraic extensions); hence we may find a copy of a rational function field within the given field. This gives the following

\begin{cor}\label{cor generalresult}
Let $L$ be a differential field that is finitely generated
over its field of constants $K\neq L$. Let $G$ be a linear algebraic group
defined over $K$ with the property that for any $r \in \N$, $G^r$ is
a differential Galois group over $K(x)$; here $K(x)$ denotes a rational function field with
derivation $d/dx$. Then $G$ is  a differential Galois
group over $L$.
\end{cor}
\begin{proof}
Let $z  \in L$ be transcendental over $K$. By Lemma \ref{change of
derivation}, we may assume that $\del(z)=1$ by replacing $\del$ with
$\del(z)^ {-1}\cdot  \del$. Then $F=K(z)$ is a differential subfield of
$L$ differentially isomorphic to $K(x)$ and the claim follows from Theorem
\ref{generalresult}.
\end{proof}

Combining Theorem \ref{result} and Corollary \ref{cor generalresult}, we
obtain our main result:

\begin{thm}\label{mainresult}
Let $K=k((t))$ be a field of Laurent series over a field $k$ of characteristic zero. Let $L/K$ be a finitely generated field extension with a non-trivial derivation $\del$ on $L$ such that $C_L=K$. Then every linear algebraic group defined over $K$ is a differential Galois group over $L$. 
\end{thm}
\end{subsection}
\end{section}

\medskip

\noindent {\bf Author Information:}

\medskip

\noindent David Harbater:
Department of Mathematics, University of Pennsylvania, Philadelphia, PA 19104-6395, USA;
email: harbater@math.upenn.edu
\medskip

\noindent Julia Hartmann:
Lehrstuhl f\"ur Mathematik (Algebra), RWTH Aachen University, 52056 Aachen, Germany;\\
current address:
Department of Mathematics, University of Pennsylvania, Philadelphia, PA 19104-6395, USA;
email: hartmann@math.upenn.edu

\medskip

\noindent Annette Maier:
Fakult\"at f\"ur Mathematik, Technische Universit\"at Dortmund, D-44221 Dortmund, Germany;
email: $\!$annette.maier@tu-dortmund.de

\medskip

\noindent The first author was supported in part by NSF grants DMS-0901164 and DMS-1265290, and NSA grant H98230-14-1-0145. 
The second author was supported by the German Excellence Initiative via RWTH Aachen University and by the German National Science Foundation (DFG).


\begin{thebibliography}{CHVdP13}

\bibitem[AM05]{amano-masuoka}
Katsutoshi Amano and Akira Masuoka.
\newblock Picard-{V}essiot extensions of {A}rtinian simple module algebras.
\newblock J. Algebra, \textbf{285} no.~2 (2005), 743--767.

\bibitem[And01]{Andre}
Yves Andr\'e.
\newblock Diff\'erentielles non-commutatives et th\'eorie de Galois diff\'erentielle ou aux diff\'erences.
\newblock Annales Scient. E.N.S. \ \textbf{34} (2001), 685--739.

\bibitem[AB94]{Anosov-Bolibrukh}
Dmitri Anosov and Andrei Bolibruch.
\newblock {\em The Riemann-Hilbert problem}. 
\newblock Aspects of Mathematics~\textbf{E22}. 
\newblock Friedr. Vieweg \& Sohn, Braunschweig, 1994.

\bibitem[Bou90]{Bourbaki}
Nicolas~Bourbaki.
\newblock {\em Algebra. {II}. {C}hapters 4--7}.
\newblock Elements of Mathematics (Berlin). Springer, Berlin, 1990.

\bibitem[BS64]{BorelSerre}
Armand~Borel and Jean-Pierre Serre.
\newblock Th\'eor\`emes de finitude en cohomologie galoisienne.
\newblock Comment. Math. Helv.\ \textbf{39} (1964), 111--164.

\bibitem[Car62]{Cartier}
Pierre Cartier.
\newblock Groupes alg\'ebriques et groupes formels.
\newblock Colloq. Th\'eorie des Groupes Alg\'ebriques (Bruxelles, 1962),  87--111.

\bibitem[CT00]{CT00}
Jean-Louis Colliot-Th\'el\`ene.
\newblock Rational connectedness and Galois covers of the projective line.
\newblock Ann.\ of Math.\ \textbf{151} (2000), 359--373.

\bibitem[CHvdP13]{CHP}
Teresa Crespo, Zbigniew Hajto and Marius van der Put.
\newblock Real and $p$-adic Picard-Vessiot fields.
\newblock 2013 manuscript. Available at arXiv:1307.2388.

\bibitem[Dyc08]{Dyc}
Tobias Dyckerhoff.
\newblock The inverse problem of differential {G}alois theory over the field
  $\mathbb{R}(z)$.
\newblock 2008 manuscript, available at arXiv:0802.2897.

\bibitem[Eps55a]{Epstein1}
Marvin P. Epstein.
\newblock  On the theory of Picard-vessiot extensions. 
\newblock Ann. of Math.~62 (1955), 528--547.

\bibitem[Eps55b]{Epstein2}
Marvin P. Epstein.
\newblock  An existence theorem in the algebraic study of homogeneous linear ordinary differential equations. Proc. Amer. Math. Soc.~\textbf{6} (1955), 33--41.

\bibitem[Gro71]{Grothendieck}
Alexander Grothendieck.
\newblock Rev\^etements \'etales et groupe fondamental. 
\newblock S\'eminaire de G\'eom\'etrie Alg\'ebrique (SGA) 1. 
\newblock Lecture Notes in Math.~\textbf{224}, Springer-Verlag, Berlin, Heidelberg and New York, 1971.

\bibitem[Har87]{HarGCAL}
David Harbater.
\newblock Galois coverings of the arithmetic line.
\newblock In {\em Number Theory:  New York, 1984-85}.  
Springer LNM, vol.\ 1240 (1987), 165--195.

\bibitem[Har03]{HarbaterMSRI}
David Harbater.
\newblock Patching and Galois theory.  
\newblock In: {\em Galois Groups and Fundamental Groups}
\newblock (L. Schneps, ed.), 
\newblock MSRI Publications series \ \textbf{41},
   Cambridge University Press (2003), 313--424.

\bibitem[HH10]{HH}
David Harbater and Julia Hartmann.
\newblock Patching over fields.
\newblock Israel J. Math., \textbf{176} (2010), 61--107.

\bibitem[HHK11]{HHKtorsor}
David Harbater, Julia Hartmann, and Daniel Krashen.
\newblock Local-global principles for torsors over arithmetic curves.
\newblock 2011 manuscript, available at: ArXiv: 1108.3323.  To appear 
in Amer.\ J.\ Math.

\bibitem[Har05]{HartCrelle}
Julia Hartmann.
\newblock On the inverse problem in differential Galois theory.
\newblock J.\ reine angew.\ Math.\ \textbf{586} (2005), 21--44

\bibitem[Har07]{Hartmann-OW}
Julia Hartmann.
\newblock Patching and differential Galois groups (joint work with David Harbater).
\newblock In: {\em Arithmetic and Differential Galois Groups}.
\newblock Oberwolfach reports~\textbf{4}, No.2, European Mathematical Society (2007).

\bibitem[Hru02]{Hrush}
Ehud Hrushovski.
\newblock  Computing the Galois group of a linear differential equation. 
\newblock {\em Differential Galois theory (B\c{e}dlewo, 2001)}, 97--138, Banach Center Publ., vol.~58, Polish Acad.\ Sci., Warsaw, 2002.

\bibitem[Kov69]{kovacic1}
Jerald Kovacic.
\newblock The inverse problem in the Galois theory of differential fields.
\newblock Ann. of Math.~\textbf{89} (1969), 583--608.

\bibitem[Kov71]{kovacic2}
Jerald Kovacic.
\newblock On the inverse problem in the Galois theory of differential fields.
\newblock Ann. of Math.~\textbf{93} (1971), 269--284.

\bibitem[MS96]{mitschisinger1}
Claudine Mitschi and Michael Singer.
\newblock Connected Linear Groups as Differential Galois Groups. 
\newblock J.\ Algebra \ \textbf{184} (1996), 333--361. 

\bibitem[MS02]{mitschisinger2}
Claudine Mitschi and Michael Singer.
\newblock Solvable-by-Finite Groups as Differential Galois Groups. 
\newblock Ann. Fac. Sci. Toulouse Math.~(6)\ \textbf{11/3} (2002), 403--423. 

\bibitem[MB01]{MB01}
Laurent Moret-Bailly.
\newblock Construction de rev\^etements de courbes point\'ees.
\newblock J.\ Algebra \textbf{240} (2001), 505--534.

\bibitem[Oor66]{Oort}
Frans Oort.
\newblock Algebraic group schemes in characteristic zero are reduced. 
\newblock Invent. Math.~\textbf{2} (1966), 79--80. 

\bibitem[Ple08]{Plemelj}
Josip Plemelj.
\newblock Riemannsche Funktionenscharen mit gegebener
Monodromiegruppe. 
\newblock Monatsh. Math. Phys.~\textbf{19} (1908), no.~1,
211--246.

\bibitem[Pop96]{Pop}
Florian Pop.
\newblock Embedding problems over large fields.
\newblock Ann.\ of Math.\ \textbf{144} (1996), 1--34.

\bibitem[Sei56]{seidenberg}
Abraham Seidenberg.
\newblock Contribution to the {P}icard-{V}essiot theory of homogeneous linear
  differential equations.
\newblock Amer. J. Math.\ \textbf{78} (1956), 808--818.

\bibitem[Sin93]{singer}
Michael Singer.
\newblock Moduli of Linear Differential Equations on the Riemann Sphere with Fixed Galois Groups.
\newblock Pac. J. of Math. \ \textbf{106}(2) (1993), 343--395. 

\bibitem[Spr09]{Springer}
Tonny Albert Springer.
\newblock {\em Linear algebraic groups}.
\newblock Birkh\"auser Boston Inc., second edition, 2009.

\bibitem[TT79]{tretkoff}
Carol Tretkoff and Marvin Tretkoff.
\newblock Solution of the Inverse Problem of Differential Galois Theory in the Classical Case.
\newblock Amer. J.~Math.\ \textbf{101}, No. 6 (1979), 1327--1332.

\bibitem[vdPS03]{singervanderput}
Marius van~der Put and Michael~F.\ Singer.
\newblock {\em Galois theory of linear differential equations}.
\newblock Springer, Berlin, 2003.

\bibitem[V{\"o}l96]{Voelklein}
Helmut V{\"o}lklein.
\newblock {\em Groups as {G}alois groups: An introduction}. Volume~53 of {\em Cambridge Studies
  in Advanced Mathematics}.
\newblock Cambridge University Press, Cambridge, 1996.


\end{thebibliography}
\end{document}